\documentclass[preprint,12pt]{article}
\usepackage{lineno,hyperref}
\usepackage{epsfig}
\usepackage{multirow}
\usepackage{epstopdf}
\usepackage{graphics}
\usepackage{graphicx}
\usepackage{makecell}
\usepackage{subfigure}
\usepackage{amssymb}
\usepackage{amsmath}
\usepackage{mathrsfs}
\usepackage{caption}
\usepackage{colortbl}
\usepackage{amsfonts}
\usepackage{bm}
\usepackage{float}
\usepackage{cite}
\usepackage{booktabs}

\newtheorem{proof}{Proof}
\newtheorem{theorem}{Theorem}[section]

\newtheorem{prop}{Proposition}[section]
\newtheorem{cor}{Corollary}

\newtheorem{algorithm}{\textbf{Algorithm}}
\counterwithin{figure}{section}
\numberwithin{figure}{section}
\numberwithin{table}{section}
\numberwithin{equation}{section}

\usepackage[top=1.2in, bottom=1.2in, left=1in, right=1in]{geometry}

\begin{document}
\title{Wiener chaos expansion for stochastic Maxwell equations driven by Wiener process}
 \author{
        Lihai Ji\footnotemark[2] \footnotemark[3], ~Kuan Xue\footnotemark[4] ~and Liying Zhang$^{\ast}$\footnotemark[4]\\
         {\small \footnotemark[2] Institute of Applied Physics and Computational Mathematics, Beijing 100094, China.} \\
         {\small \footnotemark[3] Shanghai Zhangjiang Institute of Mathematics, Shanghai 201203, China.}\\ 
          {\small \footnotemark[4] School of Mathematical Science, China University of Mining and Technology (Beijing),}\\
         {\small Beijing 100083, China.}
         }
       \maketitle
    \footnotetext{$^{\ast}$Corresponding author }

\begin{abstract}
 A novel and efficient algorithm based on the Wiener chaos expansion is proposed for the stochastic Maxwell equations driven by Wiener process. The proposed algorithm can reduce the original stochastic system to the deterministic case and separate the randomness in the computation. Therefore, it can yield a significant improvement of efficiency and lead to less computational errors compared to the Monte Carlo method, since the statistics of the solution can be solved directly without repeating over many realizations. In particular, the proposed algorithm could inherit the multi-symplecticity. Numerical experiments are dedicated to performing the efficiency and accuracy of the Wiener chaos expansion algorithm.\\


\end{abstract}
	
\section{Introduction}
\label{Introduction}
Stochastic Maxwell equations provide the foundation in stochastic electromagnetism and statistical radiophysics (see e.g., \cite{RefRKT1989}). It is of special importance to develop efficient numerical methods for simulating stochastic Maxwell equations in large scale and long time computations. In the last decade, various numerical methods for stochastic Maxwell equations have been developed, analyzed, and tested in the literature. For instance, see \cite{RefKai,RefChen2021,RefSSX2023} for the finite element methods, \cite{RefHJZC2017} for the wavelet collocation method, \cite{RefCHJ2019a,RefCHJ2019b,RefCCHS2020} for the temporally semi-discrete schemes, \cite{RefCHJL2025} for the ergodic approximation, \cite{Refhai,RefSSX2022,RefCHZ2016} for the stochastic multi-symplectic methods. For a detail description of the numerical methods for stochastic Maxwell equations as well as its implementation and applications, we refer the readers to the review paper \cite{RefZCHJ2019}, and the lecture notes \cite{RefCHJ2023}.

As we all know, the computing cost is huge for the multi-dimensional stochastic Maxwell equations due to the high dimension in space and the longtime computation requirements. To solve this issue, a few numerical methods have been developed. For example, \cite{RefHJZC2017} constructed an efficient and energy-conserving numerical approximation for the three-dimensional stochastic Maxwell equations with multiplicative noise by utilizing the wavelet collocation method in space. \cite{RefCHJ2023} introduced the operator splitting technique to simulate the stochastic Maxwell equations with additive noise in order to reduce the computational cost. However, to accurately estimate the statistical properties of the stochastic Maxwell equations, a substantial number of realizations have to be computed, which imposed more computational cost. This motivates us to design a more promising and a highly efficient algorithm to reduce the computational cost and improve efficiency.

Recently, there has been growing interest in the study of the Wiener chaos expansion (WCE) algorithm to the stochastic partial differential equations (PDEs), see for example stochastic Burgers equation \cite{RefHLRZ2006}, stochastic Navier--Stokes equation \cite{RefMR2004}, stochastic advection-diffusion-reaction equations \cite{RefZRTK2012}. In this algorithm, the deterministic effects can be separated from the randomness in the stochastic PDEs. Consequently, it leads to a system of deterministic PDEs for the expansion coefficients, which is referred to as its propagator. The main statistics (such as mean, covariance, and higher-order statistical moments) can be calculated by simple formulas involving only the coefficients of the propagator without any use of the Monte Carlo (MC) method. Recently, we note that \cite{RefBAZC2009} developed the WCE algorithm for the two-dimensional stochastic Helmholtz equation excited by a spatially incoherent source. However, as far as we know, there are no known results about the numerical approximation of time-dependent stochastic Maxwell equations based on the WCE algorithm, even for the Wiener process case. We emphasize that our main aim is to present the WCE algorithm for the time-dependent three-dimensional stochastic Maxwell equations driven by Wiener process.

The rest of the paper is organized as follows. In section \ref{Preliminary} we introduce the stochastic Maxwell equations considered in the paper and revisit the WCE definition employed in \cite{RefHLRZ2006}. In section \ref{sec:WCE} we apply the WCE algorithm to the stochastic Maxwell equations and present some theoretical results. In section 4 we compare the WCE-based algorithm with algorithms exploiting the Monte Carlo method and illustrate the accuracy and capability of the proposed algorithm. Concluding remarks are given in section 5.
%%%%%%%%%%%%%%%%
%%%%%%%%%%%%%%%%
\section{Preliminaries}
\label{Preliminary}
This section presents the notations and basic results for stochastic Maxwell equations, including the regularity of the solution, the stochastic multi-symplectic structure, the energy conservation property. Simultaniously, a brief description of the WCE definition will be given.
\subsection{Stochastic Maxwell equations}
\label{sub_stochastic MEs}
We first collect notations used throughout this paper. Fix $T>0$ and let $(\Omega,\mathcal{F},\{\mathcal{F}_t\}_{0\leq t\leq T},\mathbb{P})$ be a complete probability space.  Moreover, $W_1,W_2$ are two one-dimensional independent standard Wiener processes. Let $\Theta\subset\mathbb{R}^3$ be an open, bounded, and Lipschitz domain with boundary $\partial \Theta$, and $\bf{n}$ be the unit outward normal vector on $\partial\Theta$. Denote by ${\bf{E}}=(E_1,E_2,E_3)^\top$ and ${\bf{H}}=(H_1,H_2,H_3)^\top$ the electric field and magnetic filed, respectively. Let ${\bf{e}}=(1,1,1)^\top$, ${\bf x}=(x,y,z)$ and  denote by $|\cdot|$ the Euclidean norm, by $\langle~\cdot~\rangle$ the Euclidean inner product, and by $\mathcal{E}$ the expectation. 
	
We work with the real Hilbert space $\mathbb{H}=L^2(\Theta)^3\times L^2(\Theta)^3$, endowed with the inner product
	\[
\left\langle \begin{pmatrix}
	{\bf E}_1\\{\bf H}_1
\end{pmatrix},~ \begin{pmatrix}
	{\bf E}_2\\{\bf H}_2
\end{pmatrix}\right\rangle_{\mathbb H}=\int_{\Theta}\big({\bf E}_1\cdot {\bf E}_2
+{\bf H}_1\cdot{\bf H}_2\big){\rm d}{\bf x}
\]
for all ${\bf E}_1, {\bf H}_1,{\bf E}_2,{\bf H}_2\in L^2(\Theta)^3$, and the norm
\[
\left\|\begin{pmatrix}
	{\bf E}\\{\bf H}
\end{pmatrix}\right\|_{\mathbb H}=\left[\int_{\Theta}\left(|{\bf E}|^2+|{\bf H}|^2\right){\rm d}{\bf x}\right]^{1/2},\quad \forall~{\bf E}, {\bf H}\in  L^2(\Theta)^3.
\]

Consider the following three-dimensional stochastic Maxwell equations driven by additive noise in the It\^o sense
\begin{equation}\label{2.1}	
\left \{
\begin{split}
		&{\rm d}{\bf E}-\nabla\times{\bf H} {\rm d}t=\sigma_{1}{\bf{e}}{\rm d}W_1(t),\quad (t,{\bf x})\in(0,T]\times\Theta,\\[1.5mm]
		&{\rm d}{\bf H}+\nabla\times{\bf E} {\rm d}t=\sigma_{2}{\bf{e}}{\rm d}W_2(t),\quad (t,{\bf x})\in(0,T]\times\Theta,\\[1.5mm]
		&{\bf E}(0,{\bf x})={\bf E}_0({\bf x}),\quad{\bf H}(0,{\bf x})={\bf H}_0({\bf x}),\quad {\bf x}\in\Theta,\\[1.5mm]
		&{\bf n}\times{\bf E}={\bf 0},\quad (t,{\bf x})\in(0,T]\times\partial\Theta,
\end{split}\right.
\end{equation}
where the positive constants $\sigma_{1},\sigma_{2}$ measure the size of the noises. 

Let $Z=({\bf H}^\top,{\bf E}^\top)^\top$ and $Z_0=({\bf H}_0^\top,{\bf E}_0^\top)^\top$. In a similar way to the proof of \cite[Theorem 2.4]{RefCHJ2023}, we can establish the $\mathbb{H}$-regularity of the solution of \eqref{2.1}.
\begin{prop}\label{prop2.1}
Assume that $Z_0$ is an $\mathcal{F}_0$-measurable $\mathbb{H}$-valued random variable satisfying $\|Z_0\|_{L^p(\Omega,\mathbb{H})}\leq\infty$ for some $p\geq2$, then there exists a positive constant $C=C(p,T)$ such that the solution of \eqref{2.1} satisfies
	$$\mathcal{E}\Big[\sup_{0\leq t\leq T}\|Z(t)\|_{L^p(\Omega,\mathbb{H})}\Big]\leq C\Big(1+\mathcal{E}\big[\|Z_0\|_{\mathbb{H})}\big]\Big).$$
\end{prop}
This regularity result of the solution of stochastic Maxwell equations plays an important role in deriving the WCE algorithm, which will be introduced in the following section.

Furthermore, by denoting $S_{1}(Z)=-\sigma_{1}\langle{\bf e},{\bf H}\rangle$ and $S_{2}(Z)=\sigma_{2}\langle{\bf e},{\bf E}\rangle$, the system \eqref{2.1} can be written as
\begin{equation}\label{2.1.1}
	F{\rm d}Z+\big(K_1Z_{x}+K_2Z_{y}+K_3Z_{z}\big){\rm d}t=\nabla_Z S_{1}(Z)\circ{\rm d}W_{1}+\nabla_Z S_{2}(Z)\circ{\rm d}W_{2},
\end{equation}
where `$\circ$' indicates that the system above holds in Stratonovich integral sense, and skew-symmetric matrices $F,K_i~(i=1,2,3)$ are given by
\begin{align*}
	F=\left(\begin{array}{ccc}
		0&-I_{3\times3}\\
		I_{3\times3}&0
	\end{array}\right),\quad
	K_{i}=\left(\begin{array}{ccc}
		\mathcal{D}_i&0\\
		0&\mathcal{D}_i
	\end{array}\right),\quad i=1,2,3,
\end{align*}
with $I_{3\times3}$ being the identity matrix and
\begin{equation*}
	\mathcal{D}_{1}=\left(\begin{array}{ccccccc}
		0&0&0\\[1mm]
		0&0&-1\\[1mm]
		0&1&0\\[1mm]
	\end{array}\right),\quad
	\mathcal{D}_{2}=\left(\begin{array}{ccccccc}
		0&0&1\\[1mm]
		0&0&0\\[1mm]
		-1&0&0\\[1mm]
	\end{array}\right),\quad
	\mathcal{D}_{3}=\left(\begin{array}{ccccccc}
		0&-1&0\\[1mm]
		1&0&0\\[1mm]
		0&0&0\\[1mm]
	\end{array}\right).
\end{equation*}

The proof of the following result is very similar to that of \cite[Theorem 3.2]{RefCHJ2023}, so we omit it here.
\begin{prop}
Stochastic Maxwell equations \eqref{2.1} possesses the stochastic multi-symplectic conservative law
\begin{equation*}
{\rm d}\omega+\big(\partial_{x}\kappa_{1}+\partial_{y}\kappa_{2}+\partial_{z}\kappa_{3}\big){\rm d}t=0,\quad {\mathbb P}\text{-}a.s.,
\end{equation*}
where $\omega=\frac{1}{2}dZ\wedge FdZ$, $\kappa_{i}=\frac{1}{2}dZ\wedge K_{i}dZ$, $i=1,2,3$ are the differential 2-forms associated with the skew-symmetric matrices $F$ and $K_{i}$, respectively.
\end{prop}

For the stochastic Maxwell equations \eqref{2.1}, the energy is denoted by 
\begin{equation}
\Phi \left( t\right) = \int_\Theta\big(|{\bf E}(t,{\bf x})|^2+|{\bf H}(t,{\bf x})|^2\big){\rm d}{\bf x}. 
\end{equation}
The averaged energy $\mathcal{E}\left[\Phi \left( t\right)\right]$ is not preserved, but it satisfies the following evolution property.
\begin{theorem}\label{theorem2.3}
For any $t\in[0,T]$, there exists a constant $\gamma=3(\sigma_1^2+\sigma_2^2)|\Theta|$ such that the averaged energy of the solution of \eqref{2.1} satisfies the following linear growth law
		$$\mathcal{E}\left[\Phi \left(t\right)\right]=\mathcal{E}\left[\Phi \left(0\right)\right]+\gamma t,$$
where $|\Theta|$ is the volume of domain $\Theta$.
\end{theorem}
\begin{proof}
The proof is based on the application of It\^o formula to the following functions
		\begin{equation*}
			F_1({\bf E}(t))=\int_\Theta|{\bf E}(t,{\bf x})|^2{\rm d}{\bf x},\quad F_2({\bf H}(t))=\int_\Theta|{\bf H}(t,{\bf x})|^2{\rm d}{\bf x},\quad t\in[0,T],
		\end{equation*} 
		respectively.
		Due to $F_1$ and $F_2$ are Fr\'echet derivable, the derivatives of $F_1({\bf E})$ and $F_2({\bf H})$ along directions $\phi$ and $(\phi,\varphi)$ are respectively, as follows 
		\begin{equation}\label{frechet_der}
			\begin{split}
				DF_1({\bf E})(\phi)&=2\int_\Theta\langle {\bf E},\phi\rangle{\rm d}{\bf x},\quad D^2F_1({\bf E})(\phi,\varphi)=2\int_\Theta\langle \varphi,\phi\rangle{\rm d}{\bf x},\\[1mm]
				DF_2({\bf H})(\phi)&=2\int_\Theta\langle {\bf H},\phi\rangle{\rm d}{\bf x},\quad D^2F_2({\bf H})(\phi,\varphi)=2\int_\Theta\langle \varphi,\phi\rangle{\rm d}{\bf x}.
			\end{split}
		\end{equation}
		Applying the It\^o formula to $F_1({\bf E}(t))$, we obtain
		\begin{equation}\label{ito_for}
			\begin{split}
				F_1({\bf E}(t))=&F_1({\bf E}_0)+\int_0^tDF_1({\bf E}(s))(\sigma_{1}{\bf{e}}{\rm d}W_1(s))\\[1mm]
				&+\int_0^t\big(DF_1({\bf E}(s))(\nabla\times {\bf H}(s))+\frac12{\rm Tr}\big[D^2F_1({\bf E}(s))(\sigma_{1}{\bf{e}})(\sigma_{1}{\bf{e}})^\ast\big]\big){\rm d}s
			\end{split}
		\end{equation}
		Substituting \eqref{frechet_der} into \eqref{ito_for} leads to
		\begin{equation}\label{ito_for_2}
			\begin{split}
				F_1({\bf E}(t))=&F_1({\bf E}_0)+2\int_0^t\int_\Theta\big\langle {\bf E}(s),\sigma_{1}{\bf{e}}{\rm d}W_1(s)\big\rangle {\rm d}{\bf x}\\[1mm]
				&+\int_0^t\int_\Theta\big[2\big\langle {\bf E}(s),\nabla\times {\bf H}(s)\big\rangle+3\sigma_{1}^2\big]{\rm d}{\bf x}{\rm d}s.
			\end{split}
		\end{equation}
		Similarly, we apply It\^o formula to $F_2({\bf H}(t))$ and get
		\begin{equation}\label{ito_for_3}
			\begin{split}
				F_2({\bf H}(t))=&F_2({\bf H}_0)+2\int_0^t\int_\Theta\langle {\bf H}(s),{\bf{e}}\sigma_{2}dW_2(s)\rangle{\rm d}{\bf x}\\[1mm]
				&+\int_0^t\int_\Theta\big[-2\big\langle{\bf H}(s),\nabla\times {\bf E}(s)\big\rangle+3\sigma_{2}^2\big]{\rm d}{\bf x}{\rm d}s.
			\end{split}
		\end{equation}
		Summing \eqref{ito_for_2} and \eqref{ito_for_3} yields
		\begin{equation*}\label{ito_for_23}
			\begin{split}
				F_1({\bf E}(t))+F_2({\bf H}(t))=&F_1({\bf E}_0)+F_2({\bf H}_0)++3(\sigma_1^2+\sigma_2^2)|\Theta|t\\[1mm]
				&+2\int_0^t\int_\Theta\big[\big\langle {\bf E}(s),\sigma_{1}{\bf{e}}{\rm d}W_1(s)\big\rangle +\big\langle {\bf H}(s),\sigma_{2}{\bf{e}}{\rm d}W_2(s)\big\rangle\big]{\rm d}{\bf x}
			\end{split}
		\end{equation*}
		due to the Green formula and the PEC boundary condition ${\bf n}\times{\bf E}|_{\partial \Theta}={\bf 0}$. The assertion follows from applying expectation on the above equation.
\end{proof}
\hfill $\Box$
%%%%%%%%%%%
\subsection{Wiener chaos expansion}
\label{sub_WCE}
In this subsection, we will revisit the WCE definition and introduce the Cameron--Martin theorem. We refer to \cite{Luo,RefHLRZ2006,RefMR2004,RefZRTK2012} and references therein for more details.

For any complete orthonormal system $m_{p}(t),p=1,2,\cdots$ in the Hilbert space $L^{2}([0,T])$, we define
\begin{align*}
\xi_p^k=\int_{0}^{T}m_p(s){\rm d}W_k(s),\quad p=1,2,\cdots,
\end{align*}
with $W_k,~k=1,2,\cdots$ being one-dimensional independent standard Wiener processes. It is well known that $\xi_p^k,~p, k=1,2,\cdots$ are independent identically distributed standard Gaussian random variables due to the It\^o isometry. Moreover, the Wiener process $W_k(t)$ has the following Fourier expansion
\begin{align}\label{Wt}
    W_k(t)=\sum_{p=1}^{\infty}\xi_p^k\int_{0}^{t}m_p(s){\rm d}s,\quad 0\leq t \leq T.
\end{align}

Throughout this paper, we choose the basis $m_p(t)$, $p=1,2,\cdots$ as
$$m_1(t)=\frac{1}{\sqrt{T}},\quad m_p(t)=\sqrt{\frac{2}{T}}\cos\left(\frac{(p-1)\pi t}{T}\right),\quad p=2,3,\cdots,\quad 0\leq t\leq T.$$
 Denote by $\mathcal{J}$ the set of multi-indices $\alpha=\{\alpha_{k,p}\}_{k,p\geq1}$ of finite length $|\alpha|=\sum_{k,p=1}^\infty\alpha_{k,p}$, i.e.,
\begin{align*}
  \mathcal{J}=\big\{\alpha=(\alpha_{k,p}, k,p\geq1),\quad \alpha_{k,p}\in\{0,1,2,\cdots\},\quad|\alpha|<\infty\big\}.
\end{align*}
Here $k$ denotes the number of Wiener processes and $p$ the number of Gaussian random variables approximating each Wiener process. For $\alpha\in \mathcal{J}$, we define the Wick polynomials by the tensor products
\begin{equation}
	T_{\alpha}:=\prod_{\alpha}\frac{H_{\alpha_{k,p}}(\xi_{p}^{k})}{\sqrt{\alpha_{k,p}!}},\quad \alpha\in\mathcal{J},
\end{equation}
where $H_{{\alpha_{k,p}}}$ is the $\alpha_{k,p}$-th Hermite polynomial
$$H_{\alpha_{k,p}}(x)=(-1)^{\alpha_{k,p}}e^{x^2/2}\frac{{\rm d}^{\alpha_{k,p}}}{{\rm d}x^{\alpha_{k,p}}}e^{-x^2/2}.$$
Actually, it is not hard to verify that $\{T_{\alpha}\}$ is a complete orthonormal system in $L^2(\Omega,\mathcal{F},\mathbb{P})$. 

The following result, often referred to as the Cameron-Martin theorem given in \cite{RefCam}, forms the basic theoretical foundation of the WCE algorithm.
\begin{theorem}\label{WCE1}
Assume that for any fixed ${\bf x}$ and $t\in[0,T]$, the solution $u(t, \bf x)$ of the following general stochastic PDE
\begin{equation*}
{\rm d}u(t,{\bf x})=\mathcal{L}u(t,{\bf x}){\rm d}t+\sum_{k=1}^{\infty}\sigma_k{\rm d}W_k(t),~~ (t,{\bf x})\in(0,T]\times \Theta;\quad u(0,{\bf x})=u_0({\bf x}),~~ {\bf x}\in \Theta
\end{equation*}
satisfies $\mathcal{E}\big[|u(t,{\bf x})|^2\big]<\infty$ , then $u(t,{\bf x})$ has the following Fourier--Hermite expansion (also called WCE)
\begin{equation*}
u(t,{\bf x})=\sum_{\alpha\in\mathcal{J}}u_{\alpha}(t,{\bf x})T_{\alpha},\quad u_{\alpha}(t,{\bf x})=\mathcal{E}\big[u(t,{\bf x})T_{\alpha}\big].
\end{equation*}
Furthermore, the first two statistical moments of $u(t,{\bf x})$ are given by
\begin{equation*}
\mathcal{E}\big[u(t,{\bf x})\big]=u_{0}(t,{\bf x}),\quad \mathcal{E}\big[u^{2}(t,{\bf x})\big]=\sum_{\alpha\in\mathcal{J}}\big|u_{\alpha}(t,{\bf x})\big|^{2}.
\end{equation*}
\end{theorem}

Moreover, the theorem for the nonlinear problem is also given in \cite{RefCam}.
\begin{theorem}\label{WCE2}
Assume that $u(t,\bf x)$ and $v(t,\bf x)$ have Fourier--Hermite expansions
\begin{equation*}
\begin{split}
   u(t,{\bf x})=\sum_{\alpha\in\mathcal{J}}u_{\alpha}(t,{\bf x})T_{\alpha},&\quad u_{\alpha}(t,{\bf x})=\mathcal{E}\big[u(t,{\bf x})T_{\alpha}\big],\\
   v(t,{\bf x})=\sum_{\alpha\in\mathcal{J}}v_{\alpha}(t,{\bf x})T_{\alpha},&\quad v_{\alpha}(t,{\bf x})=\mathcal{E}\big[v(t,{\bf x})T_{\alpha}\big].
\end{split}
\end{equation*}
If $u(t,{\bf x})v(t,\bf x)$ satisfies $\mathcal{E}\left[ \left | u(t,{\bf x})v(t,{\bf x})\right|^2\right]<\infty$, then $u(t,{\bf x})v(t,{\bf x})$ has Fourier--Hermite expansion
\begin{equation*}
u(t,{\bf x})v(t, {\bf x})= \sum_{\alpha \in \mathcal{J}}\left [\sum_{\rho \in \mathcal{J}}\sum_{0 \le \beta \le \alpha} G(\alpha,\beta,\rho) u_{\alpha-\beta+\rho}(t,{\bf x})v_{\beta+\rho}(t,{\bf x}) \right]T_{\alpha},
\end{equation*}  
where $G(\alpha,\beta,\rho)=\big[ \binom{\alpha}{\rho}\binom{\beta+\rho}{\rho}\binom{\alpha-\beta+\rho}{\rho} \big]^{\frac{1}{2}}$.
\end{theorem}  

According to Theorem \ref{WCE2}, the third and fourth order statistical moments of $u(t,{\bf x})$ can be given by
\begin{equation*}
\begin{split}
\mathcal{E}\big[u^{3}(t,{\bf x})\big]&=\sum_{\alpha \in \mathcal{J}}\left [\sum_{\rho \in \mathcal{J}}\sum_{0 \le \beta \le \alpha} G(\alpha,\beta,\rho) u_{\alpha-\beta+\rho}(t,{\bf x})u_{\beta+\rho}(t,{\bf x}) \right]u_{\alpha}(t,{\bf x}),\\
\mathcal{E}\big[u^{4}(t,{\bf x})\big]&=\sum_{\alpha \in \mathcal{J}}\left [\sum_{\rho \in \mathcal{J}}\sum_{0 \le \beta \le \alpha} G(\alpha,\beta,\rho) u_{\alpha-\beta+\rho}(t,{\bf x})u_{\beta+\rho}(t,{\bf x}) \right]^2.
\end{split}
\end{equation*}

In the following sections we will develop the WCE for stochastic Maxwell equations and apply it to compute the statistical moments of the solution.

\section{WCE algorithm for stochastic Maxwell equations}
\label{sec:WCE}
To obtain the WCE algorithm for the stochastic Maxwell equations, we rewrite \eqref{2.1.1} in integral form, it yields for any $t\in[0,T]$ that
\begin{equation}\label{3.1}
FZ(t,{\bf x})=FZ_0({\bf x})-\int_{0}^{t}\Big[K_1Z_{x}(s,{\bf x})+K_2Z_{y}(s,{\bf x})+K_3Z_{z}(s,{\bf x})\Big]{\rm d}s+\sum_{k=1}^2\nabla_Z S_{k}(Z)\circ W_{k}(t).
\end{equation}
It follows from Proposition \ref{prop2.1} and Theorem \ref{WCE1} that the solution of the stochastic Maxwell equations (\ref{2.1.1}) is defined as
\begin{align*}
  Z(t,{\bf x})=\sum_{\alpha\in\mathcal{J}} Z_{\alpha}(t,{\bf x})T_{\alpha},\quad0\leq t\leq T,
\end{align*}
with 
\begin{equation*}
	T_{\alpha}=\prod_{k=1}^2\prod_{p=1}^{\infty}\frac{H_{\alpha_{k,p}}(\xi_{p}^{k})}{\sqrt{\alpha_{k,p}!}},\quad \alpha\in\mathcal{J}.
\end{equation*}

Now we are in the position to determine the WCE coefficients $Z_\alpha$. To this end, multiplying both sides of \eqref{3.1} by $T_{\alpha}$ and taking expectation
\begin{equation}\label{3.8}
\begin{split}
   FZ_{\alpha}(t,{\bf x})=FZ_0({\bf x}){\bf 1}_{\{|\alpha|=0\}}-&\int_{0}^{t}\Big[K_1\partial_xZ_{\alpha}(s,{\bf x})+K_2\partial_yZ_{\alpha}(s,{\bf x})+K_3\partial_zZ_{\alpha}(s,{\bf x})\Big]{\rm d}s\\[1.5mm]
   &+\sum_{k=1}^2\nabla_{Z_{\alpha}} S_{k}(Z_{\alpha})\circ\mathcal{E}\big[W_{k}(t)T_{\alpha}\big],
  \end{split}
\end{equation}
where ${\bf 1}_{\{|\alpha|=0\}}$ is the indicator function defined as
\begin{align*}
 {\bf 1}_{\{|\alpha|=0\}}=
  \left\{
\begin{array}{cc}
1,&\mathrm{if}~|\alpha|=0,\\[2mm]
0,&\mathrm{otherwise}.
\end{array}\right.
\end{align*}
For the last term of the right hand side of \eqref{3.8}, we get from \eqref{Wt} that
\begin{equation}\label{3.9}
	\begin{split}
\mathcal{E}\big[W_{k}(t)T_\alpha\big]=&\sum_{p=1}^{\infty}\mathcal{E}\big[\xi_{p}^{k}T_{\alpha}\big]\int_{0}^{t}m_{p}(s){\rm d}s\\
=&\sum_{p=1}^{\infty}{\bf 1}_{\{\alpha_{k,p}=1,~|\alpha|=1\}}\int_{0}^{t}m_p(s){\rm d}s,\quad k=1,2,
\end{split}
\end{equation}
in view of the Hermite polynomial $H_{1}(x)=x$ and the orthogonality of $T_{\alpha}$. Here, the indicator function ${\bf 1}_{\{\alpha_{k,p}=1,~|\alpha|=1\}}$ with $k=1,2$ is defined by
\begin{align*}
{\bf 1}_{\{\alpha_{k,p}=1,~|\alpha|=1\}}=
	\left\{
	\begin{array}{cc}
		1,&\mathrm{if}~\alpha_{k,p}=1,~|\alpha|=1,\\[2mm]
		0,&\mathrm{otherwise}.~~~~~~~~~~~~
	\end{array}\right. 
\end{align*}

Substituting \eqref{3.9} into \eqref{3.8}, we arrive at
\begin{equation*}
	\begin{split}
		FZ_{\alpha}(t,{\bf x})=FZ_0({\bf x}){\bf 1}_{\{|\alpha|=0\}}-&\int_{0}^{t}\Big[K_1\partial_xZ_{\alpha}(s,{\bf x})+K_2\partial_yZ_{\alpha}(s,{\bf x})+K_3\partial_zZ_{\alpha}(s,{\bf x})\Big]{\rm d}s\\[1.5mm]
		&+\sum_{k=1}^2\sum_{p=1}^{\infty}\nabla_{Z_{\alpha}} S_{k}(Z_{\alpha})\circ{\bf 1}_{\{\alpha_{k,p}=1,~|\alpha|=1\}}\int_{0}^{t}m_p(s){\rm d}s.
	\end{split}
\end{equation*}
Therefore, the WCE coefficient $Z_\alpha$ satisfies the following deterministic system
\begin{equation}\label{3.11}
F\partial_tZ_{\alpha}+K_1\partial_xZ_{\alpha}+K_2\partial_yZ_{\alpha}+K_3\partial_zZ_{\alpha}=\sum_{k=1}^2\sum_{p=1}^{\infty}\nabla_{Z_{\alpha}} S_{k}(Z_{\alpha}){\bf 1}_{\{\alpha_{k, p}=1,~|\alpha|=1\}}m_p(t).
\end{equation}
For the sake of simplicity we shall present the WCE algorithm for the deterministic initial datum case and remark that
\begin{equation}\label{3.5}
Z_{\alpha}(0,{\bf x})=Z_{0}({\bf x}){\bf 1}_{\{|\alpha|=0\}},\quad {\bf x}\in\Theta.
\end{equation}

Rewriting above equations \eqref{3.11} and \eqref{3.5} through its components ${\bf E}_\alpha,{\bf H}_{\alpha}$, we obtain
\begin{equation}\label{3.6}
\left \{	
\begin{split}
		&\partial_t{\bf E}_\alpha-\nabla\times{\bf H}_\alpha=\sigma_{1}{\bf{e}}\sum_{p=1}^{\infty}{\bf 1}_{\{\alpha_{1,p}=1,~|\alpha|=1\}}m_p(t),\quad (t,{\bf x})\in(0,T]\times\Theta,\\
		&\partial_t{\bf H}_\alpha+\nabla\times{\bf E}_\alpha=\sigma_{2}{\bf{e}}\sum_{p=1}^{\infty}{\bf 1}_{\{\alpha_{2,p}=1,~|\alpha|=1\}}m_p(t),\quad (t,{\bf x})\in(0,T]\times\Theta,\\[1.5mm]
		&{\bf E}_\alpha(0,{\bf x})={\bf E}_0({\bf x}){\bf 1}_{\{|\alpha|=0\}},\quad{\bf H}_\alpha(0,{\bf x})={\bf H}_0({\bf x}){\bf 1}_{\{|\alpha|=0\}},\quad {\bf x}\in\Theta,\\[2.5mm]
		&{\bf n}\times{\bf E}_\alpha={\bf 0},\quad (t,{\bf x})\in(0,T]\times\partial\Theta.
	\end{split}
\right.
\end{equation}
It is now easy to check that the WCE coefficient $Z_\alpha$ preserves the multi-symplectic conservation law which stated in the following corollary.
\begin{cor}
WCE coefficient $Z_\alpha$ satisfies a multi-symplectic Hamiltonian system
  \begin{align*}
    F\partial_tZ_{\alpha}+K_1\partial_xZ_{\alpha}+K_2\partial_yZ_{\alpha}+K_3\partial_zZ_{\alpha}=\nabla_{Z_{\alpha}}\tilde{S}(Z_{\alpha}),
  \end{align*}
  with $Z_{\alpha}=({\bf H}_{\alpha}^\top,{\bf E}_{\alpha}^\top)^\top$ and skew-symmetric matrices $F,K_1,K_2,K_3$ being the same as those in section \ref{sub_stochastic MEs}. Moreover, it possesses the multi-symplectic conservation law
  \begin{align*}
 \partial_t\omega+\partial_x\kappa_1+\partial_y\kappa_2+\partial_z\kappa_3=0,
  \end{align*}
  with $\omega=\frac12dZ_{\alpha}\wedge FdZ_{\alpha}$ and $\kappa_{i}=\frac12dZ_{\alpha}\wedge K_idZ_{\alpha}$, $i=1,2,3$.
\end{cor}
\begin{proof}
 By introducing a smooth function of the state variable $Z_{\alpha}$ $$\widetilde{S}(Z_{\alpha})=\sum_{k=1}^2\sum_{p=1}^{\infty} S_{k}(Z_{\alpha}){\bf 1}_{\{\alpha_{k,p}=1,~|\alpha|=1\}}m_p(t),$$
 where $S_{1}(Z_\alpha)=-\sigma_{1}\langle{\bf e},{\bf H}_{\alpha}\rangle$ and $S_{2}(Z_\alpha)=\sigma_{2}\langle{\bf e},{\bf E}_{\alpha}\rangle$, then the  assertions can be obtained easily.
\end{proof}
\hfill $\Box$

For the equations \eqref{3.6}, we have the following energy evolution property.
\begin{cor}
The energy of the solution of \eqref{3.6} satisfies
$$
\int_\Theta\Big(|{\bf E}_\alpha(t,{\bf x})|^2+|{\bf H}_\alpha(t,{\bf x})|^2\Big){\rm d}{\bf x}=\int_\Theta\Big(|{\bf E}_0({\bf x})|^2+|{\bf H}_0({\bf x})|^2\Big){\bf 1}_{\{|\alpha|=0\}}{\rm d}{\bf x}+\Gamma,
$$
with
\begin{equation*}
	\begin{split}
		\Gamma=&2\sigma_{1}\sum_{p=1}^{\infty}{\bf 1}_{\{\alpha_{1,p}=1,~|\alpha|=1\}}\int_0^tm_p(t)\int_\Theta \big\langle{\bf{e}},{\bf E}_\alpha\big\rangle{\rm d}{\bf x}{\rm d}t\\
		&+2\sigma_{2}\sum_{p=1}^{\infty}{\bf 1}_{\{\alpha_{2,p}=1,~|\alpha|=1\}}\int_{0}^tm_p(t)\int_\Theta \big\langle{\bf{e}},{\bf H}_\alpha\big\rangle{\rm d}{\bf x}{\rm d}t.
	\end{split}
\end{equation*}
\end{cor}
\begin{proof}
By Green’s formula and \eqref{3.6} it gets
\begin{equation}\label{3.7}
	\begin{split}
		\int_\Theta\Big(\big\langle\partial_t{\bf E}_\alpha,{\bf E}_\alpha\big\rangle+\big\langle\partial_t{\bf H}_\alpha,{\bf H}_\alpha\big\rangle\Big){\rm d}{\bf x}=&\int_\Theta\big\langle\nabla\times{\bf H}_\alpha+\sigma_{1}{\bf{e}}\sum_{p=1}^{\infty}{\bf 1}_{\{\alpha_{1,p}=1,~|\alpha|=1\}}m_p(t),{\bf E}_\alpha\big\rangle{\rm d}{\bf x}\\
		&-\int_\Theta\big\langle\nabla\times{\bf E}_\alpha-\sigma_{2}{\bf{e}}\sum_{p=1}^{\infty}{\bf 1}_{\{\alpha_{2,p}=1,~|\alpha|=1\}}m_p(t),{\bf H}_\alpha\big\rangle{\rm d}{\bf x}\\
		=&-\int_\Theta\nabla\cdot\big({\bf E}_\alpha\times{\bf H}_\alpha\big){\rm d}{\bf x}+A,
	\end{split}	
\end{equation}
with 
$$A=\sigma_{1}\sum_{p=1}^{\infty}{\bf 1}_{\{\alpha_{1,p}=1,~|\alpha|=1\}}m_p(t)\int_\Theta\big\langle{\bf{e}},{\bf E}_\alpha\big\rangle{\rm d}{\bf x}+\sigma_{2}\sum_{p=1}^{\infty}{\bf 1}_{\{\alpha_{2,p}=1,~|\alpha|=1\}}m_p(t)\int_\Theta\big\langle{\bf{e}},{\bf H}_\alpha\big\rangle{\rm d}{\bf x}.$$
It following from the PEC boundary condition. We finally complete the proof.
\end{proof}
\hfill $\Box$

Now the following part is dedicated to presenting the WCE algorithm for the stochastic Maxwell equations \eqref{2.1}. In fact, we need to truncate the equation \eqref{3.6} for the numerical implementation. We define the order of multi-index 
$$d(\alpha):=\max\big\{p\geq1:~\alpha_{k,p}>0~{\rm for ~some}~k>1\big\},$$
and the truncated set of multi-indices be denoted by
\begin{align*}
	\mathcal{J}_{N,I}=\big\{\alpha\in\mathcal{J}:~|\alpha|\leq N,~d(\alpha)\leq I\big\},
\end{align*}
where $N$ is the highest Hermite polynomial order and $I$ is the maximum number of Gaussian random variables for each Wiener process. Then the truncated WCE of stochastic Maxwell equations is given by
\begin{equation}\label{WCE_trun}
	{\bf E}_{N,I}(t,{\bf x})=\sum_{\alpha\in\mathcal{J}_{N,I}}{\bf E}_{\alpha}(t,{\bf x})T_{\alpha},\quad {\bf H}_{N,I}(t,{\bf x})=\sum_{\alpha\in\mathcal{J}_{N,I}}{\bf H}_{\alpha}(t,{\bf x})T_{\alpha}.
\end{equation}

Thus, the truncated expansion \eqref{WCE_trun} together with \eqref{3.6} gives us a constructive approximation of the solution to the stochastic Maxwell equations \eqref{2.1}. 
\begin{algorithm}\label{algorithm 1}
Step 1. Choose a complete orthonormal system $\{m_p(t)\}_{p\geq1}$; take the truncated parameters $N$ and $I$ to determine the size of the multi-index set $\mathcal{J}_{N,I}$.

Step 2. Introduce the temporal step size $\Delta t$ and the spatial mesh grid size $\Delta x$, $\Delta y$, $\Delta z$; Choose an appropriate fully discrete scheme to approximation \eqref{3.6} for $\alpha\in\mathcal{J}_{N,I}$ in domain $[0,T]\times\Theta$ and obtain the WCE coefficients ${\bf E}_{\alpha}$, ${\bf H}_{\alpha}$.

Step 3. According to \eqref{WCE_trun}, obtain the approximate solution of \eqref{2.1}.

Step 4. According to Theorem \ref{WCE1} and Theorem \ref{WCE2}, compute the statistical moments of the solution to \eqref{2.1},
\begin{equation*}
	\begin{split}
	&\mathcal{E}\big[{\bf E}(t,{\bf x})\big]={\bf E}_{\{|\alpha|=0\}}(t,{\bf x}),\quad \mathcal{E}\big[{\bf E}^{2}(t,{\bf x})\big]=\sum_{\alpha\in\mathcal{J}_{N,I}}\big|{\bf E}_{\alpha}(t,{\bf x})\big|^{2},\\
    &\mathcal{E}\big[{\bf E}^{3}(t,{\bf x})\big]=\sum_{\alpha\in\mathcal{J}_{N,I}}\big[\sum_{\rho\in\mathcal{J}_{N,I}}\sum_{0\le \beta \le \alpha}G(\alpha,\beta,\rho){\bf E}_{\alpha-\beta+\rho}(t,{\bf x}){\bf E}_{\beta+\rho}(t,{\bf x})\big]{\bf E}_{\alpha}(t,{\bf x}),\\
    &\mathcal{E}\big[{\bf E}^{4}(t,{\bf x})\big]=\sum_{\alpha\in\mathcal{J}_{N,I}}\big[\sum_{\rho\in\mathcal{J}_{N,I}}\sum_{0\le \beta \le \alpha}G(\alpha,\beta,\rho){\bf E}_{\alpha-\beta+\rho}(t,{\bf x}){\bf E}_{\beta+\rho}(t,{\bf x})\big]^{2},\\    
	&\mathcal{E}\big[{\bf H}(t,{\bf x})\big]={\bf H}_{\{|\alpha|=0\}}(t,{\bf x}),\quad \mathcal{E}\big[{\bf H}^{2}(t,{\bf x})\big]=\sum_{\alpha\in\mathcal{J}_{N,I}}\big|{\bf H}_{\alpha}(t,{\bf x})\big|^{2},\\
    &\mathcal{E}\big[{\bf H}^{3}(t,{\bf x})\big]=\sum_{\alpha\in\mathcal{J}_{N,I}}\big[\sum_{\rho\in\mathcal{J}_{N,I}}\sum_{0\le \beta \le \alpha}G(\alpha,\beta,\rho){\bf H}_{\alpha-\beta+\rho}(t,{\bf x}){\bf H}_{\beta+\rho}(t,{\bf x})\big]{\bf H}_{\alpha}(t,{\bf x}),\\
    &\mathcal{E}\big[{\bf H}^{4}(t,{\bf x})\big]=\sum_{\alpha\in\mathcal{J}_{N,I}}\big[\sum_{\rho\in\mathcal{J}_{N,I}}\sum_{0\le \beta \le \alpha}G(\alpha,\beta,\rho){\bf H}_{\alpha-\beta+\rho}(t,{\bf x}){\bf H}_{\beta+\rho}(t,{\bf x})\big]^{2}.
	\end{split}
\end{equation*}
\end{algorithm}
\section{Numerical results}
This section presents a comprehensive numerical investigation of the one-dimensional (1D), two-dimensional (2D), and three-dimensional (3D) stochastic Maxwell equations driven by Wiener processes. Our primary objective is to compute the statistical moments of the solutions and the average energy associated with the stochastic Maxwell equations. The numerical results obtained using the WCE algorithm are systematically compared with those generated by the MC method. 

Firstly, in order to describe the numerical experiments, we introduce the general difference operators employed in the spatial and temporal discretization. Let $\Delta x$, $\Delta y$ and $\Delta z$ denote the spatial mesh grid size, and $\Delta t$ represents the temporal step size. The $\left[0,T\right]\times \left[0,X\right]\times \left[0,Y\right]\times \left[0,Z\right]$ is partitioned by parallel lines, where 
\begin{equation*}
\begin{split}
&t_n=n\Delta t,~n=0,1,\cdots,N,\quad x_i=i\Delta x,~i=0,1,\cdots,I,\\
&y_j=j\Delta y,~j=0,1,\cdots,J,\quad z_k=k\Delta z,~k=0,1,\cdots,K.
\end{split}
\end{equation*}
The grid function $U^{n,i,j,k}$ is the numerical approximation to the exact solution $U(t_n,x_i,y_j,z_k)$ at these discrete points. The general difference operators  in time and space are defined as follows
\begin{equation*}
\begin{split}
&\delta_tU^{n,i,j,k}=\dfrac{U^{n+1,i,j,k}-U^{n,i,j,k}}{\Delta t}, \quad \delta_xU^{n,i,j,k}=\dfrac{U^{n,i+1,j,k}-U^{n,i-1,j,k}}{2\Delta x},\\ 
&\delta_yU^{n,i,j,k}=\dfrac{U^{n,i,j+1,k}-U^{n,i,j-1,k}}{2\Delta y}, \quad \delta_zU^{n,i,j,k}=\dfrac{U^{n,i,j,k+1}-U^{n,i,j,k-1}}{2\Delta z}.
\end{split}
\end{equation*}

Subsequently, in order to compare the accuracy between the WCE algorithm and MC method, we use the Frobenius norm to compute the relative error of the two methods and it takes the form
\begin{equation*}
err(u)=\dfrac{\left \| u_{WCE}-u_{MC} \right \|_{F} }{\left \| u_{MC} \right \|_{F} },
\end{equation*}
where $u_{WCE}$ and $u_{MC}$ denote the numerical solutions computed via the WCE  and MC, respectively.
%1-D
\subsection{1-D stochastic Maxwell equations}
In this subsection, we mainly consider the following 1-D stochastic Maxwell equations with additive noise
\begin{equation}\label{4.1}
\left\{
\begin{split}
\partial_{t}E_{1}&=-\partial_{x}H_{1}-\sigma\dot{W}(t),~~(t,x)\in[0,T]\times\Theta,\\[2mm]
\partial_{t}H_{1}&=-\partial_{x}E_{1}+\sigma\dot{W}(t),~~(t,x)\in[0,T]\times\Theta,
\end{split}\right.
\end{equation}
where $T=1$,  $\Theta =\left [0, 2\pi \right]$ and $\dot{W}(t)$ denotes the formal derivative of the Wiener process $W(t)$. Meanwhile, the equations (\ref{4.1}) satisfies periodic boundary conditions and the initial conditions are specified by
\begin{equation*}
\begin{split}
E_{1}(0,x)=\sin{x}+\cos{x},~~H_{1}(0,x)=\sin{x}-\cos{x}.
\end{split}
\end{equation*}

 Next, we will derive the discretized formulation of WCE algorithm for equations (\ref{4.1}). We assume that WCE truncated set is $\mathcal{J}_{N_1,I_1}$ and it is the set of multi-indices $\alpha=\{\alpha_{p}\}$, then discretized formulation is as follows
\begin{equation*}
\begin{split}
\delta_{t}E_{1,\alpha}^{n,i}&=-\delta_{x}H_{1,\alpha}^{n,i}-\sigma \sum_{p=1}^{N_1}I_{\{\alpha_{p}=1,~|\alpha|=1\}}m_{p}^{n},\\
\delta_{t}H_{1,\alpha}^{n,i}&=-\delta_{x}E_{1,\alpha}^{n,i}+\sigma \sum_{p=1}^{N_1}I_{\{\alpha_{p}=1,~|\alpha|=1\}}m_{p}^{n},\\
\end{split}
\end{equation*}
where $m_{p}^{n}=m_{p}(t_n)$ and $E_{1,\alpha}$, $H_{1,\alpha}$ denote the WCE coefficients of $E_1$, $H_1$, respectively. According to the Algorithm \ref{algorithm 1}, we derive the third and fourth order statistical moments of equations
\begin{equation*}
\begin{split}
\mathcal{E}\left(E_1^{n,i}\right)^3&=\sum_{\alpha\in \mathcal{J}_{N_1,I_1}}\big[ \sum_{\rho\in \mathcal{J}_{N_1,I_1}}
\sum_{0\le \beta\le \alpha}G(\alpha,\beta,\rho)
 E^{n,i}_{1,\alpha-\beta+\rho}E^{n,i}_{1,\beta+\rho}\big]E^{n,i}_{1,\alpha},\\
\mathcal{E}\left(E_1^{n,i}\right)^4&=\sum_{\alpha\in \mathcal{J}_{N_1,I_1}}\big[ \sum_{\rho\in \mathcal{J}_{N_1,I_1}}
\sum_{0\le \beta\le \alpha}G(\alpha,\beta,\rho)
 E^{n,i}_{1,\alpha-\beta+\rho}E^{n,i}_{1,\beta+\rho}\big]^2.
\end{split}
\end{equation*} 
Furthermore, the discrete averaged energy $\mathcal{E}\left[\Phi(t_n)\right]$ is defined as 
\begin{equation*}
\mathcal{E}\left[\Phi(t_n)\right]=\sum_{\alpha\in \mathcal{J}_{N_1,I_1}}\sum_{i=0}^{I}\left[\left(E_{1,\alpha}^{n,i}\right)^2+\left(H_{1,\alpha}^{n,i}\right)^2 \right]\Delta x,~~n=0,1,\cdots,N.
\end{equation*}

In the first experiment, the WCE truncated set is $\mathcal{J}_{20,2}$ and the size of noise is fixed at $\sigma=1$. The spatial mesh grid size is $\Delta x= 2\pi/200$, and the temporal step size is $\Delta t =1/1000$. We compare the numerical results of WCE algorithm with those of MC method with $20000$ realizations. Figure~\ref{pic1:4.1} and Figure~\ref{pic1:4.2} demonstrate the third and fourth order statistical moments of $E_1$ and $H_1$ at $T=1$. We observe the two results almost coincide. Table~\ref{table1} displays the relative errors of the two methods for various statistical moments. The results indicate there is comparable accuracy between the WCE algorithm and the MC methods. However, the computation times of the WCE algorithm and MC method  are $6.867$s, $502.528$s, respectively. It is obvious that the WCE algorithm havs superior efficiency than the MC method.

\begin{figure}[H]
\begin{center}
  \includegraphics[height=4.2cm,width=8cm]{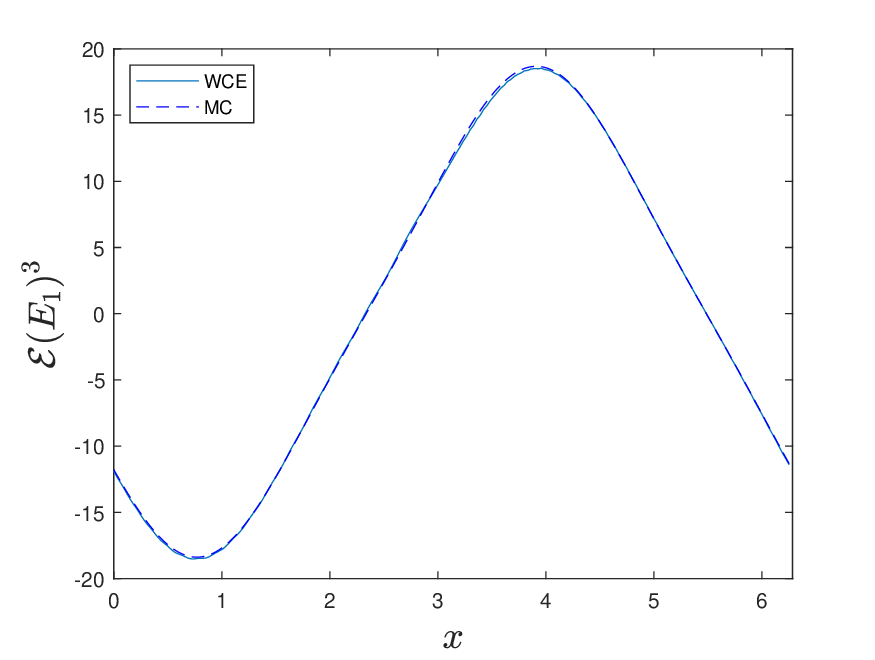}
  \includegraphics[height=4.2cm,width=8cm]{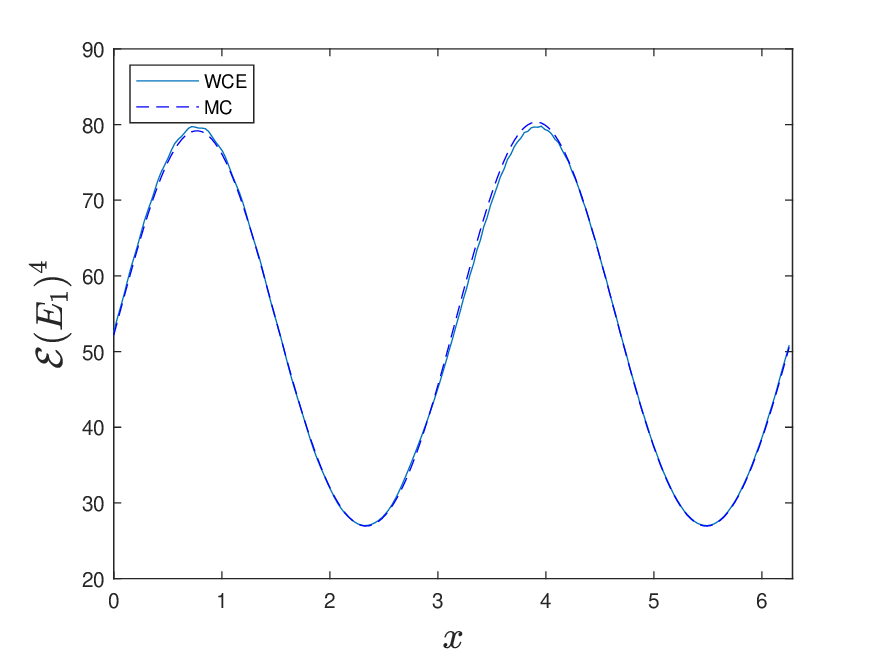}
  \caption{Third order statistical moments of $E_1$ (left) and fourth order statistical moments of $E_1$ (right) for $\sigma=1$ and $T=1$.}\label{pic1:4.1}
  \end{center}
\end{figure}
\begin{figure}[H]
\begin{center}
  \includegraphics[height=4.2cm,width=8cm]{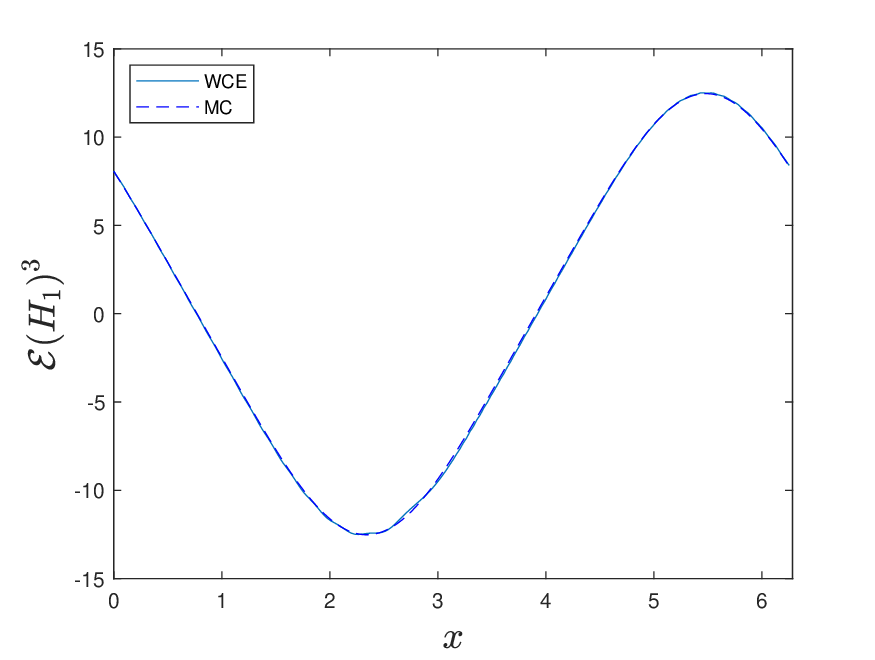}
  \includegraphics[height=4.2cm,width=8cm]{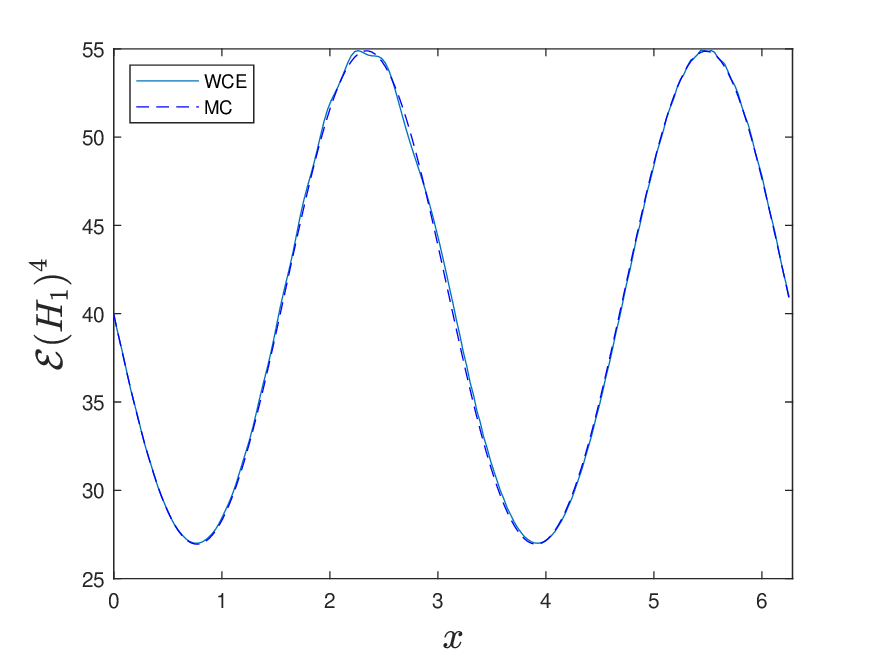}
  \caption{Third order statistical moments of $H_1$ (left) and fourth order statistical moments of $H_1$ (right) for $\sigma=1$ and $T=1$.}\label{pic1:4.2}
  \end{center}
\end{figure}
\begin{table}[H]
\centering
\fontsize{10.5pt}{25pt}
\begin{tabular}{cccccc} 
\toprule
 $u$ & $err\left[\mathcal{E}(u)\right]$& $ err\left[\mathcal{E}(u)^2\right]$ & $err\left[\mathcal{E}(u)^3\right]$ & $err\left[\mathcal{E}(u)^4\right]$\\
\midrule
$u=E_1$ &0.0096  & 0.0063 &0.0106  &0.0083\\
$u=H_1$ &0.0088  & 0.0054 &0.0105  &0.0055 \\
 \bottomrule
\end{tabular}
\caption{Relative errors of the WCE algorithm and MC method}
\label{table1}
\end{table}
 
Subsequently, we consider the temporal evolution of the average energy for varying sizes of noise $\sigma$ by $\sigma=0$, $\sigma=0.1$, $\sigma=0.5$ and $\sigma=1$. Figure~\ref{pic1:4.3} presents the average energy growth curves of the WCE algorithm, the MC method and the exact solution derived from Theorem~\ref{theorem2.3}. It is evident that the results obtained via the WCE algorithm demonstrate remarkable agreement with the exact solution. While the MC method also coincides with the exact solution, the sampling variability causes the deviations.
\begin{figure}[H]
\begin{center}
  \includegraphics[height=4.4cm,width=8cm]{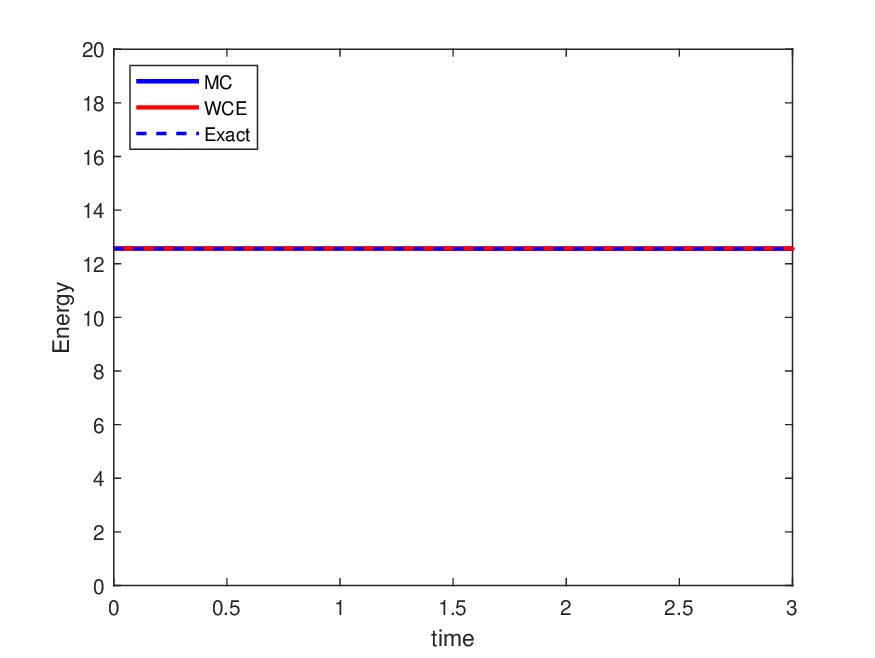}
  \includegraphics[height=4.4cm,width=8cm]{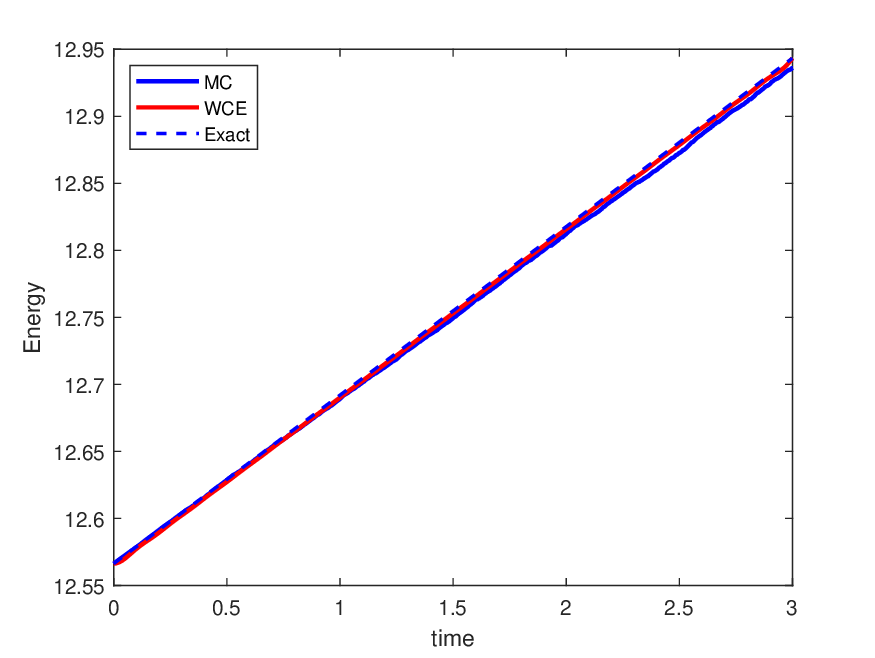}
  \includegraphics[height=4.4cm,width=8cm]{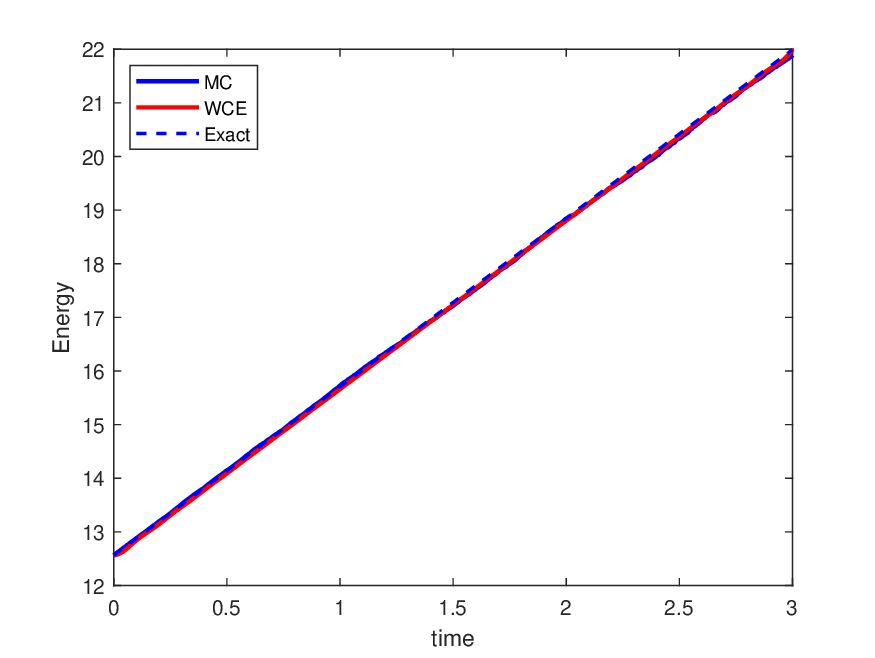}
  \includegraphics[height=4.4cm,width=8cm]{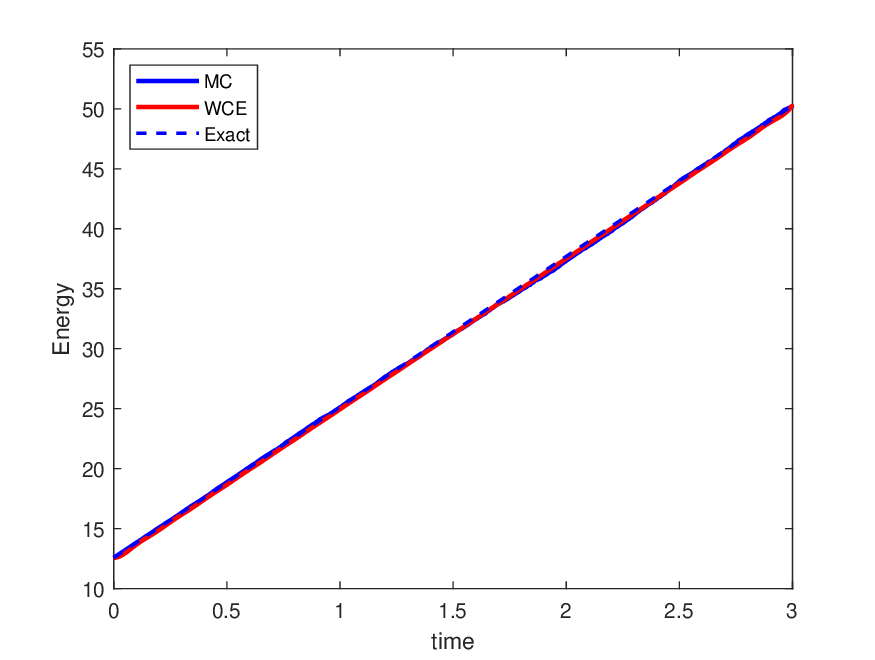}
  \caption{Averaged energy by WCE and MC for different sizes of noise $\sigma = 0$ (top left), $\sigma = 0.1$ (top right), $\sigma = 0.5$ (below left) and $\sigma = 1$ (below right).}\label{pic1:4.3}
  \end{center}
\end{figure} 
%2-D 
\subsection{2-D stochastic Maxwell equations}
In this subsection, we mainly consider the following 2-D stochastic Maxwell equations
\begin{equation}\label{4.2}
\left\{
\begin{split}
\partial_{t}E_{3}&=\partial_{x}H_{2}-\partial_{y}H_{1}+\sigma\dot{W}(t),~~(t,{\bf x})\in[0,T]\times\Theta,\\[2mm]
\partial_{t}H_{1}&=-\partial_{y}E_{3}+\sigma\dot{W}(t),~~(t,{\bf x})\in[0,T]\times\Theta,\\[2mm]
\partial_{t}H_{2}&=\partial_{x}E_{3}+\sigma\dot{W}(t),~~(t,{\bf x})\in[0,T]\times\Theta,
\end{split}\right.
\end{equation}
where $T=1$, ${\bf x}=(x,y)$ and $\Theta =\left [0, 2\pi \right]\times \left[0, 2\pi\right]$. The equations satisfies the periodic boundary conditions and the initial conditions being
\begin{equation*}
\begin{split}
E_{3}(0,x,y)=\sin{x}-\cos{y},~~H_{1}(0,x,y)=\cos{y},~~H_{2}(0,x,y)=\sin{x}.
\end{split}
\end{equation*}

Analogous to the 1-D case, the WCE truncated set is $\mathcal{J}_{N_1,I_1}$. The discretized WCE formulation for the equations (\ref{4.2}) is given by
\begin{equation*}
\begin{split}
\delta_{t}E_{3,\alpha}^{n,i,j}&=\delta_{x}H_{2,\alpha}^{n,i,j}-\delta_{y}H_{1,\alpha}^{n,i,j}+\sigma \sum_{p=1}^{N_1}I_{\{\alpha_{p}=1,~|\alpha|=1\}}m_{p}^{n},\\
\delta_{t}H_{1,\alpha}^{n,i,j}&=-\delta_{x}E_{3,\alpha}^{n,i,j}+\sigma \sum_{p=1}^{N_1}I_{\{\alpha_{p}=1,~|\alpha|=1\}}m_{p}^{n},\\
\delta_{t}H_{2,\alpha}^{n,i,j}&=\delta_{x}E_{3,\alpha}^{n,i,j}+\sigma \sum_{p=1}^{N_1}I_{\{\alpha_{p}=1,~|\alpha|=1\}}m_{p}^{n}.\\
\end{split}
\end{equation*}
Focusing on the electric field component $E_3$ as an example, its third and fourth order statistical moments are computed according to the following expressions
\begin{equation*}
\begin{split}
\mathcal{E}\left(E_3^{n,i,j}\right)^3&=\sum_{\alpha\in \mathcal{J}_{N_1,I_1}}\big[ \sum_{\rho\in \mathcal{J}_{N_1,I_1}}
\sum_{0\le \beta\le \alpha}G(\alpha,\beta,\rho)
 E^{n,i,j}_{3,\alpha-\beta+\rho}E^{n,i,j}_{3,\beta+\rho}\big]E^{n,i,j}_{3,\alpha},\\
\mathcal{E}\left(E_3^{n,i,j}\right)^4&=\sum_{\alpha\in \mathcal{J}_{N_1,I_1}}\big[\sum_{\rho\in \mathcal{J}_{N_1,I_1}}
\sum_{0\le \beta\le \alpha}G(\alpha,\beta,\rho)
 E^{n,i,j}_{3,\alpha-\beta+\rho}E^{n,i,j}_{3,\beta+\rho}\big]^2.
\end{split}
\end{equation*} 
The discrete averaged energy $\mathcal{E}\left[\Phi(t_n)\right]$ satisfies 
\begin{equation*}
\mathcal{E}\left[\Phi(t_n)\right]=\sum_{\alpha\in \mathcal{J}_{N_1,I_1}}\sum_{i=0}^{I}\sum_{j=0}^{J}\left[\left(E_{3,\alpha}^{n,i,j}\right)^2+\left(H_{1,\alpha}^{n,i,j}\right)^2+\left(H_{2,\alpha}^{n,i,j}\right)^2  \right]\Delta x\Delta y,~~n=0,1,\cdots,N.
\end{equation*}

For the numerical simulations of the 2-D equations, the WCE truncated set is  $\mathcal{J}_{20, 3}$ and the size of noise is set to $\sigma=1$. We take the spatial mesh grid size $\Delta x=\Delta y=2\pi/60$, and the temporal step size is $\Delta t=1/1000$. The  MC method  is performed with $10000$ realizations for comparison. Figure~\ref{pic2:4.1} and Figure~\ref{pic2:4.2} demonstrate the 3D plot of third and fourth order statistical moments of $E_3$. Table~\ref{table2} demonstrates the relative errors between the two algorithms for various statistical moments. The results demonstrate that the WCE algorithm achieves accuracy comparable to that of the  MC method. Then we consider the computation times of two algorithm, the computation times of the WCE algorithm is $301.896$s and the computation times of  MC method is $1299.118$s. It is obvious that the WCE algorithm also has high computational efficiency in the 2-D equations.
\begin{figure}[H]
\begin{center}
  \includegraphics[height=5cm,width=6.5cm]{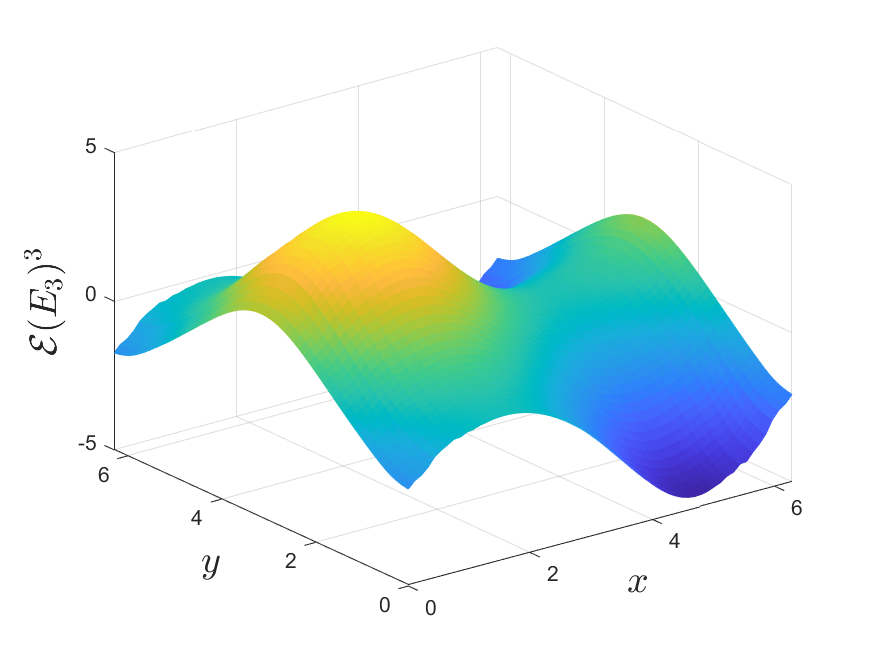}
  \includegraphics[height=5cm,width=6.5cm]{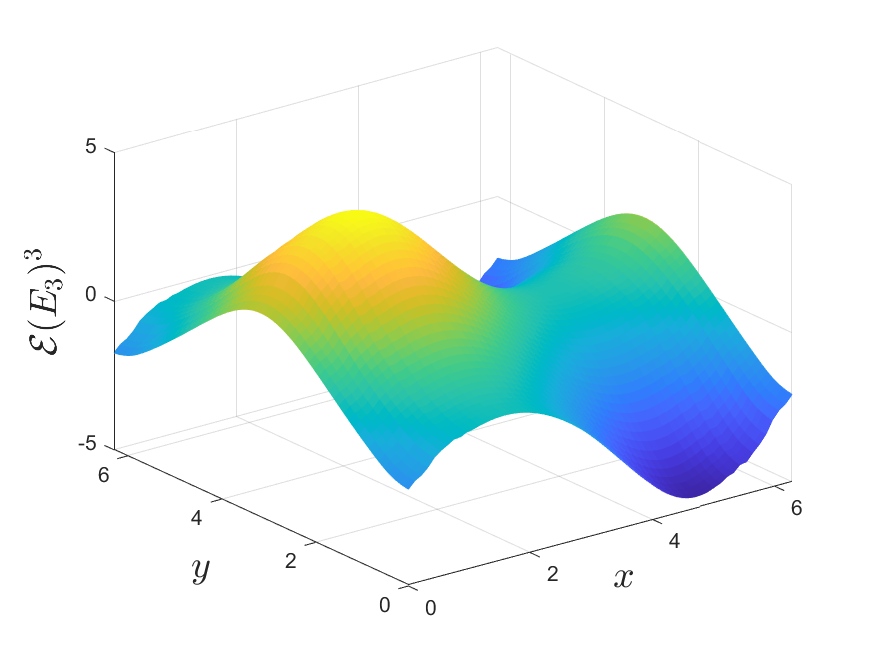}
  \caption{Third order statistical moments of of $E_3$ by MC (left) and WCE (right) for $\sigma=1$ and $T=1$.}\label{pic2:4.1}
  \end{center}
\end{figure}
\begin{figure}[H]
\begin{center}
  \includegraphics[height=5cm,width=6.5cm]{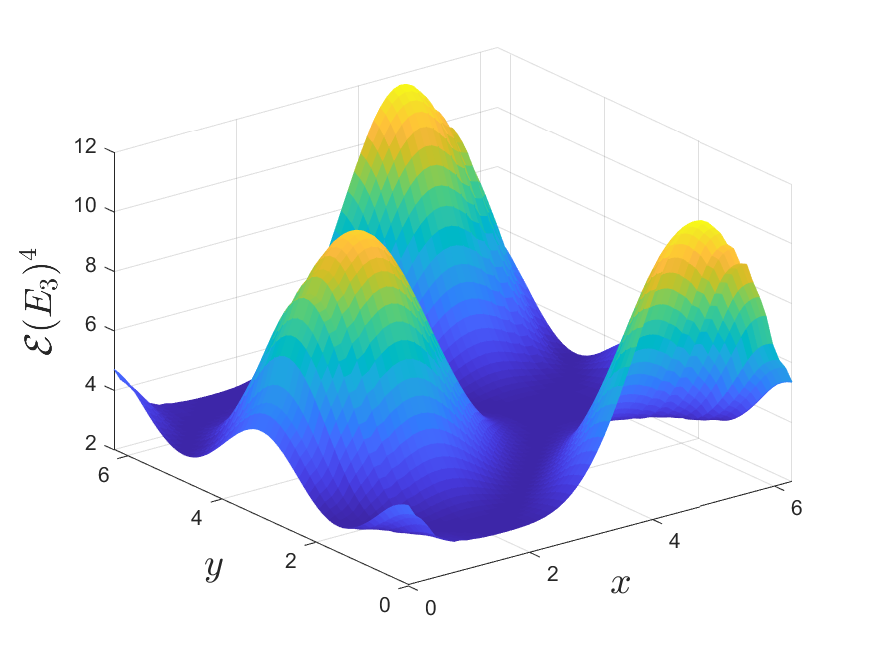}
  \includegraphics[height=5cm,width=6.5cm]{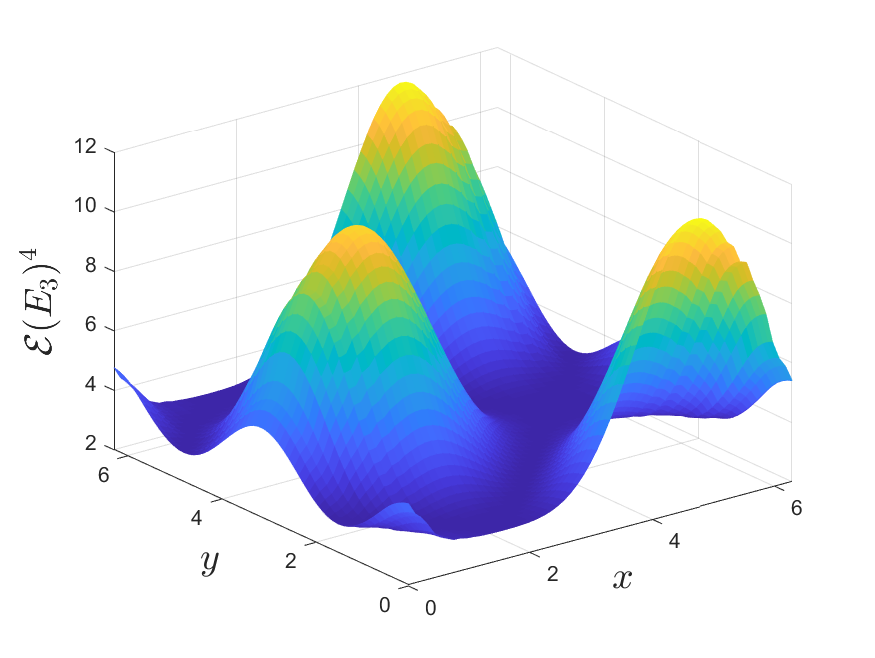}
  \caption{Fourth order statistical moments of of $E_3$ by MC (left) and WCE (right) for $\sigma=1$ and $T=1$.}\label{pic2:4.2}
  \end{center}
\end{figure}
\begin{table}[H]
\centering
\fontsize{10.5pt}{25pt}
 \begin{tabular}{cccccc} 
\toprule
 $u$ & $err\left[\mathcal{E}(u)\right]$& $ err\left[\mathcal{E}(u)^2\right]$ & $err\left[\mathcal{E}(u)^3\right]$ & $err\left[\mathcal{E}(u)^4\right]$\\
\midrule
$u=E_3$ & 0.0102 & 0.0165 & 0.0183  &0.0415\\
 \bottomrule
\end{tabular}
\caption{Relative errors of the WCE algorithm and  MC method}
\label{table2}
\end{table}

‌Basing on the previous numerical conditions, we now compute the temporal evolution of the average energy for different $\sigma$. Figure~\ref{pic2:4.3} presents the average energy growth curves. Comparing the results from the WCE algorithm and  MC method with the exact solution, it is obvious that the WCE results almost coincide with those of the exact solution. While the  MC method satisfies the linear growth, we still observe the minor deviations between  MC method and the exact solution.
\begin{figure}[H]
\begin{center}
  \includegraphics[height=4.4cm,width=8cm]{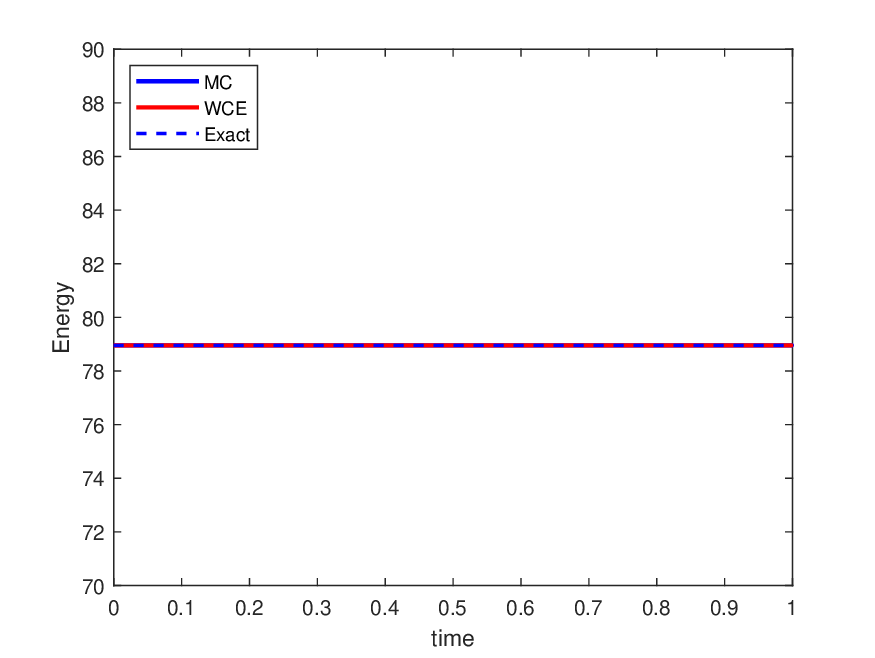}
  \includegraphics[height=4.4cm,width=8cm]{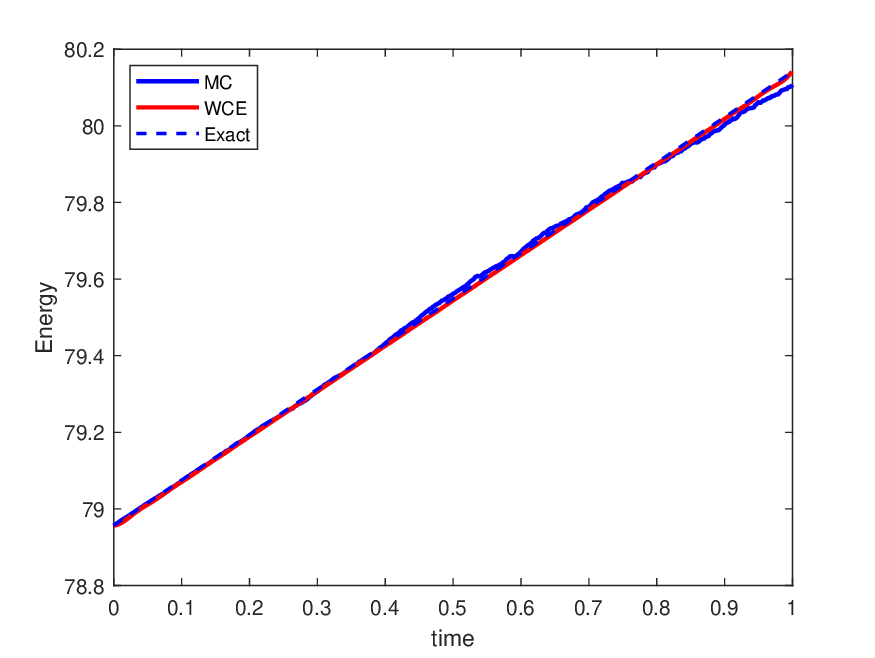}
  \includegraphics[height=4.4cm,width=8cm]{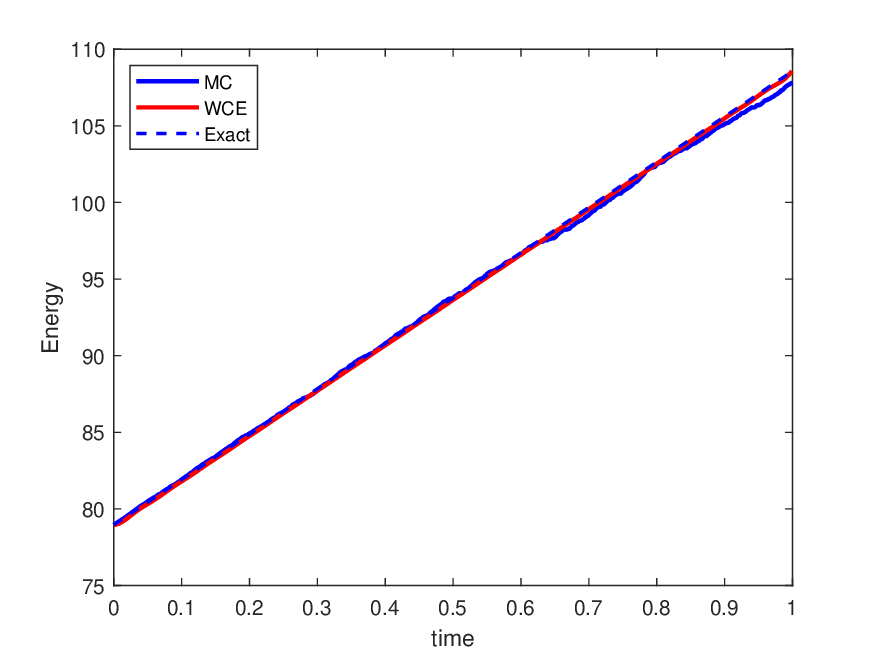}
  \includegraphics[height=4.4cm,width=8cm]{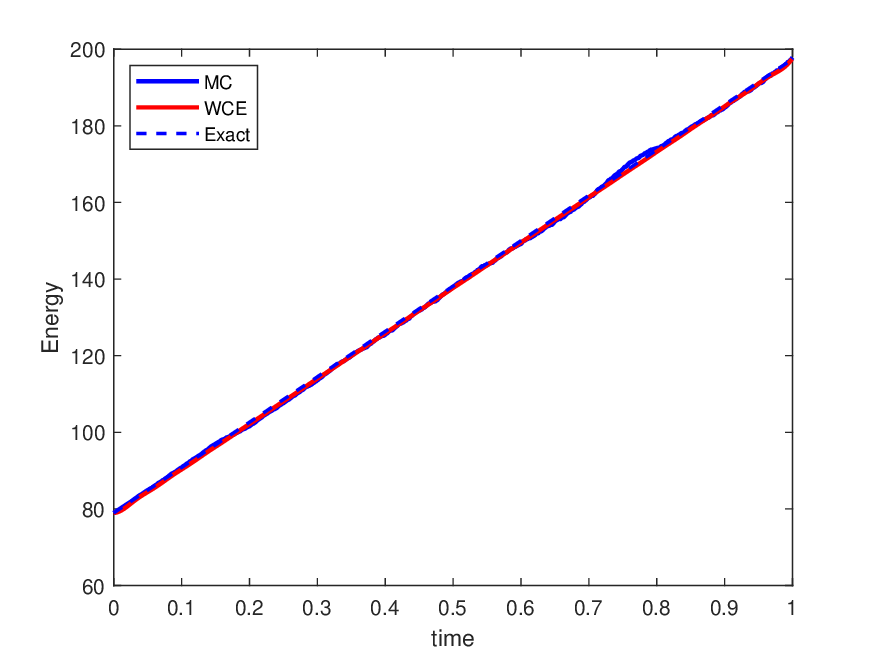}
  \caption{Averaged energy by WCE and MC for different sizes of noise $\sigma = 0$ (top left), $\sigma = 0.1$ (top right), $\sigma = 0.5$ (below left) and $\sigma = 1$ (below right).}\label{pic2:4.3}
  \end{center}
\end{figure}
%3-D
\subsection{3-D stochastic Maxwell equations}
Finally, we consider the following 3-D stochastic Maxwell equations with additive noise
\begin{equation}\label{4.3}
\left \{ \begin{split}
\partial_{t}{\bf E}&=~~\nabla \times {\bf H}+\sigma{\bf e}\dot{W},~~(t,{\bf x}) \in \left[0,T \right]\times \Theta,\\[2mm]  
\partial_{t}{\bf H}&=-\nabla \times {\bf E}+\sigma{\bf e}\dot{W},~~(t,{\bf x}) \in \left[0,T \right]\times \Theta,
\end{split}\right.
\end{equation} 
where $T=1$, $\Theta =\left[0 ,1 \right]^3$,
${\bf E}=(E_1,E_2,E_3)^\top$ and ${\bf H}=(H_1,H_2,H_3)^\top$. The equations satisfies the periodic boundary conditions and the initial conditions being
\begin{equation*}
\begin{split}
E_{1}(0,x,y,z)&=\frac{5}{\sqrt{14}}\cos(\pi x)\sin(2\pi y)\sin(-3\pi z),H_{1}(0,x,y,z)=\sin(\pi x)\cos(2\pi y)\cos(-3\pi z),\\
E_{2}(0,x,y,z)&=\frac{-4}{\sqrt{14}}\sin(\pi x)\cos(2\pi y)\sin(-3\pi z),H_{2}(0,x,y,z)=\cos(\pi x)\sin(2\pi y)\cos(-3\pi z),\\
E_{3}(0,x,y,z)&=\frac{-1}{\sqrt{14}}\sin(\pi x)\sin(2\pi y)\sin(-3\pi z),H_{3}(0,x,y,z)=\cos(\pi x)\cos(2\pi y)\sin(-3\pi z).
\end{split}
\end{equation*} 

Analogous to the lower-dimensional cases, the discretized WCE algorithm for the equations (\ref{4.3}) expands to the following expressions for each field component
\begin{equation*}
\begin{split}
\delta_{t}E_{1,\alpha}^{n,i,j,k}&=\delta_{y}H_{3,\alpha}^{n,i,j,k}-\delta_{z}H_{2,\alpha}^{n,i,j,k}+\sigma \sum_{p=1}^{N_1}I_{\{\alpha_{p}=1,|\alpha|=1\}}m_{p}^{n},\\
\delta_{t}E_{2,\alpha}^{n,i,j,k}&=\delta_{z}H_{1,\alpha}^{n,i,j,k}-\delta_{x}H_{3,\alpha}^{n,i,j,k}+\sigma \sum_{p=1}^{N_1}I_{\{\alpha_{p}=1,|\alpha|=1\}}m_{p}^{n},\\
\delta_{t}E_{3,\alpha}^{n,i,j,k}&=\delta_{x}H_{2,\alpha}^{n,i,j,k}-\delta_{y}H_{1,\alpha}^{n,i,j,k}+\sigma \sum_{p=1}^{N_1}I_{\{\alpha_{p}=1,|\alpha|=1\}}m_{p}^{n},\\
\delta_{t}H_{1,\alpha}^{n,i,j,k}&=\delta_{z}E_{2,\alpha}^{n,i,j,k}-\delta_{y}E_{3,\alpha}^{n,i,j,k}+\sigma \sum_{p=1}^{N_1}I_{\{\alpha_{p}=1,|\alpha|=1\}}m_{p}^{n},\\
\delta_{t}H_{2,\alpha}^{n,i,j,k}&=\delta_{x}E_{3,\alpha}^{n,i,j,k}-\delta_{z}E_{1,\alpha}^{n,i,j,k}+\sigma \sum_{p=1}^{N_1}I_{\{\alpha_{p}=1,|\alpha|=1\}}m_{p}^{n},\\
\delta_{t}H_{3,\alpha}^{n,i,j,k}&=\delta_{y}E_{1,\alpha}^{n,i,j,k}-\delta_{x}E_{2,\alpha}^{n,i,j,k}+\sigma \sum_{p=1}^{N_1}I_{\{\alpha_{p}=1,|\alpha|=1\}}m_{p}^{n}.\\
\end{split}
\end{equation*}
Taking the $E_1$ for the example, the third and fourth order statistical moments of $E_1$ are computed as follows
\begin{equation*}
\begin{split}
\mathcal{E}\left(E_1^{n,i,j,k}\right)^3&=\sum_{\alpha\in \mathcal{J}_{N_1,I_1}}\big[ \sum_{\rho\in \mathcal{J}_{N_1,I_1}}
\sum_{0\le \beta\le \alpha}G(\alpha,\beta,\rho)
 E^{n,i,j,k}_{1,\alpha-\beta+\rho}E^{n,i,j,k}_{1,\beta+\rho}\big]E^{n,i,j,k}_{1,\alpha},\\
\mathcal{E}\left(E_1^{n,i,j,k}\right)^4&=\sum_{\alpha\in \mathcal{J}_{N_1,I_1}}\big[\sum_{\rho\in \mathcal{J}_{N_1,I_1}}
\sum_{0\le \beta\le \alpha}G(\alpha,\beta,\rho)
 E^{n,i,j,k}_{1,\alpha-\beta+\rho}E^{n,i,j,k}_{1,\beta+\rho}\big]^2.
\end{split}
\end{equation*} 
The discrete averaged energy is defined by 
\begin{equation*}
\mathcal{E}\left[ \Phi(t_n)\right]=\sum_{\alpha\in \mathcal{J}_{N_1,I_1}}\sum_{i=0}^{I}\sum_{j=0}^{J}\sum_{k=0}^{K}\left(\left|{\bf E}_{\alpha}^{n,i,j,k}\right|^2+\left|{\bf H}_{\alpha}^{n,i,j,k}\right|^2  \right)\Delta x\Delta y\Delta z,~~n=0,1,\cdots,N,
\end{equation*}
where ${\bf E}_{\alpha}^{n,i,j,k}=\left( E_{1,\alpha}^{n,i,j,k},E_{2,\alpha}^{n,i,j,k},E_{3,\alpha}^{n,i,j,k}\right)^{\top}$ and ${\bf H}_{\alpha}^{n,i,j,k}=\left( H_{1,\alpha}^{n,i,j,k},H_{2,\alpha}^{n,i,j,k},H_{3,\alpha}^{n,i,j,k}\right)^{\top}$.
 
Considering the significantly increased computational demands in three dimensions, we mainly focus on the third and fourth order statistical moments of $E_1(1, y, z)$, $E_1(x, 1, z)$ and $E_1( x, y, 1)$ at $T=1$ and $\sigma=1$. To reduce computational cost, we take the spatial mesh grid size $\Delta x = \Delta y = \Delta z = 1/50$, and the temporal step size is $\Delta t=1/1000$. The WCE truncated set is $\mathcal{J}_{12,2}$.  MC method is performed with $1000$ realizations. Figure~\ref{pic3:4.1} and Figure~\ref{pic3:4.2} demonstrate the 3D plot of third and fourth order statistical moments of $E_1(1, y, z)$, $E_1(x, 1, z)$ and $E_1( x, y, 1)$. 
Table~\ref{table3} displays the relative errors between the two algorithms for various statistical moments. Consistent with the lower-dimensional findings, the results demonstrate the two algorithms have comparable accuracy. The computation times of the WCE algorithm and  MC method are $3645.2665$s, $4396.8753$s, respectively. It is worth that we pay attention to this phenomenon. Due to hardware constraints, the number of MC method cannot be substantially increased. While this results in relatively small discrepancies between the two algorithms for computation time, a limited sample size reduces the accuracy achievable by  MC method. The process shows a significant disadvantage of the  MC method for high-dimensional stochastic problems.
\begin{figure}[H]
\begin{center}
  \includegraphics[height=3.3cm,width=4.5cm]{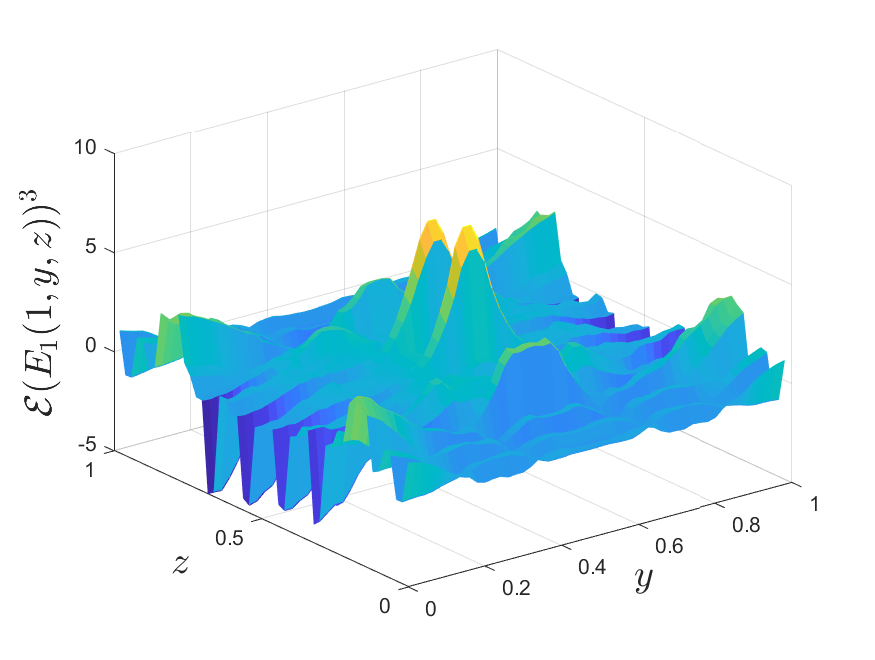}
    \includegraphics[height=3.3cm,width=4.5cm]{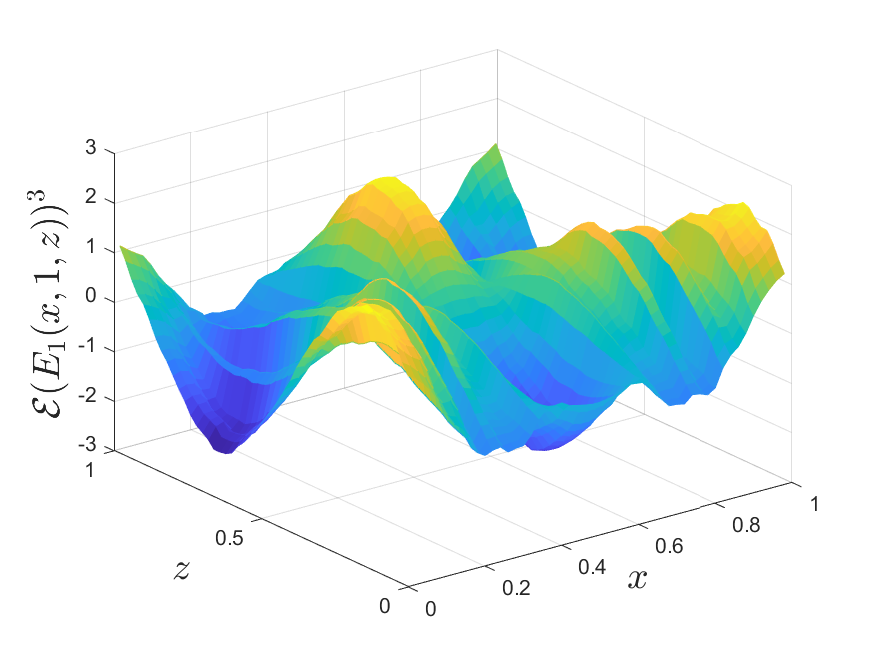}
  \includegraphics[height=3.3cm,width=4.5cm]{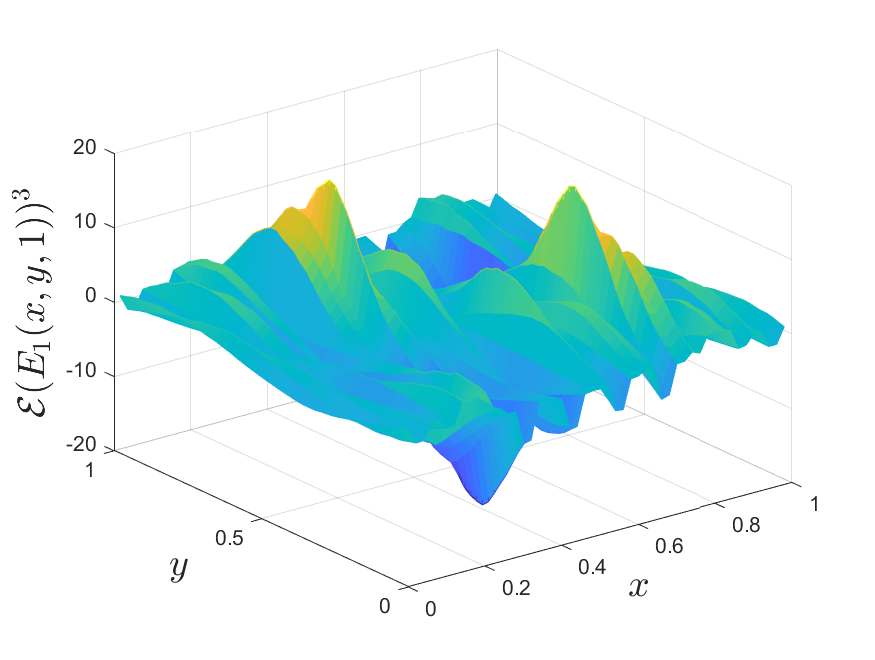}
  \includegraphics[height=3.3cm,width=4.5cm]{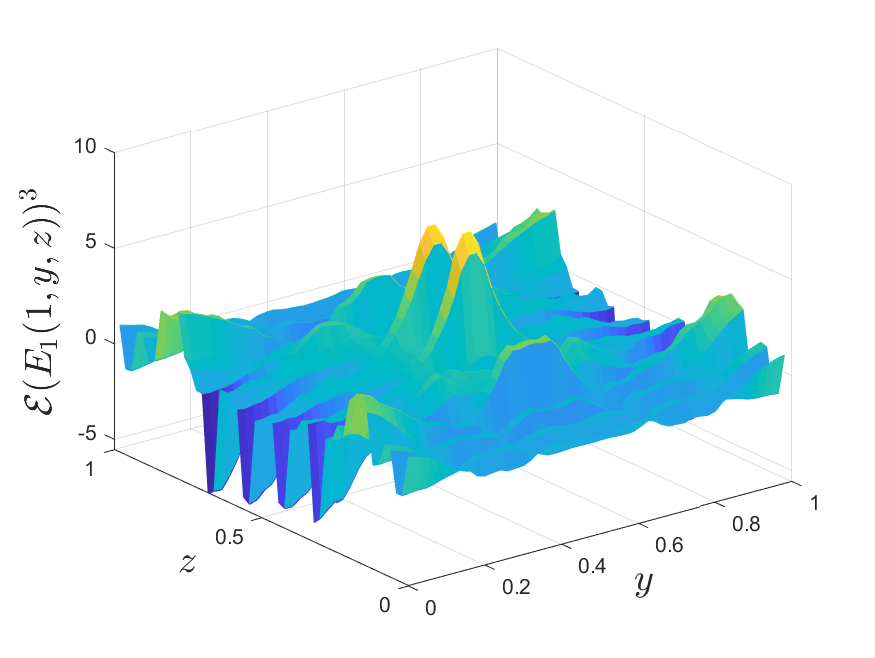}
    \includegraphics[height=3.3cm,width=4.5cm]{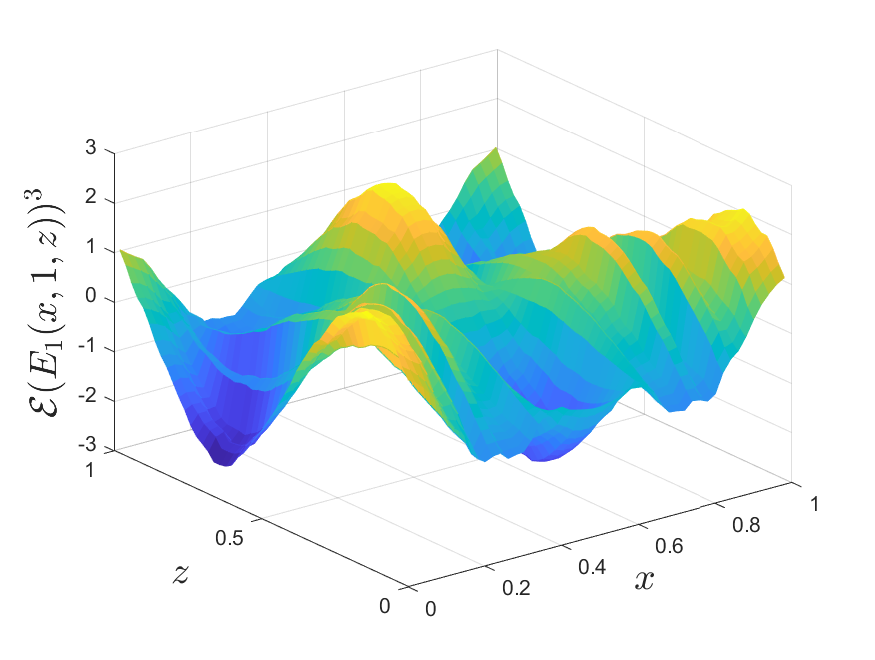}
  \includegraphics[height=3.3cm,width=4.5cm]{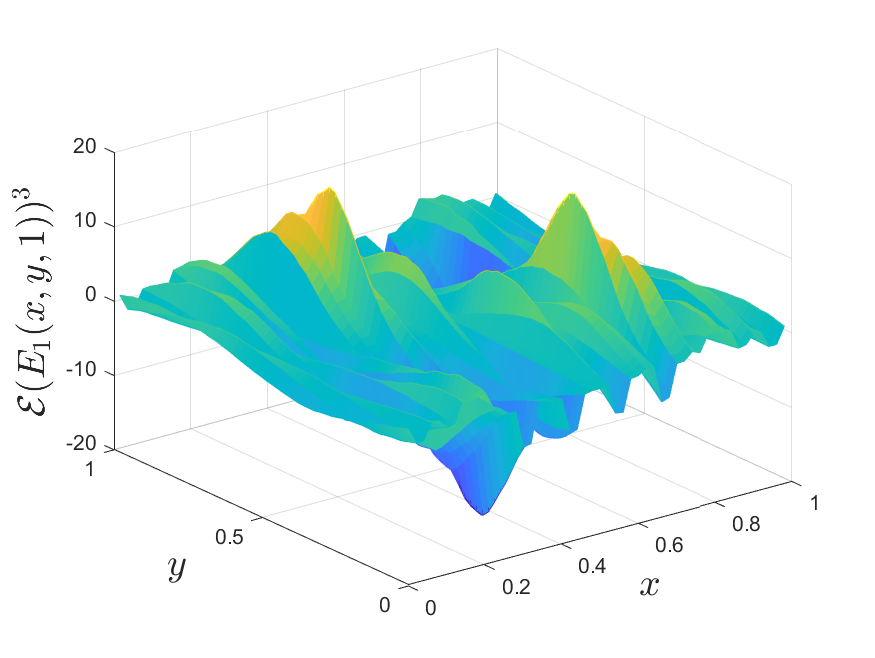}
  \caption{Third order statistical moments of $E_1$ when $x=1$, $y=1$ and $z=1$ 
  respectively by MC (top) and WCE (below).}\label{pic3:4.1}
  \end{center}
\end{figure}
\begin{figure}[H]
\begin{center}
  \includegraphics[height=3.3cm,width=4.5cm]{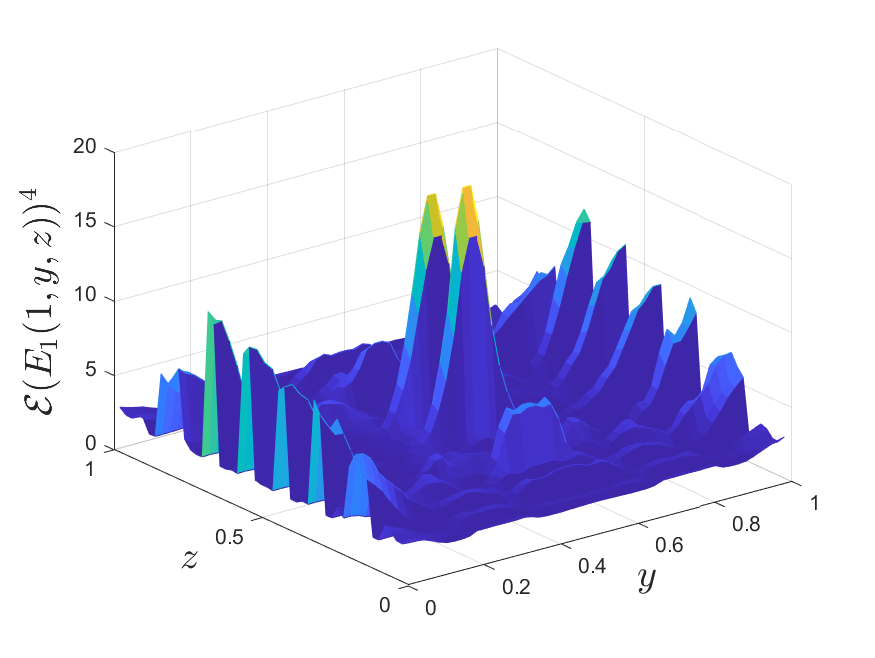}
    \includegraphics[height=3.3cm,width=4.5cm]{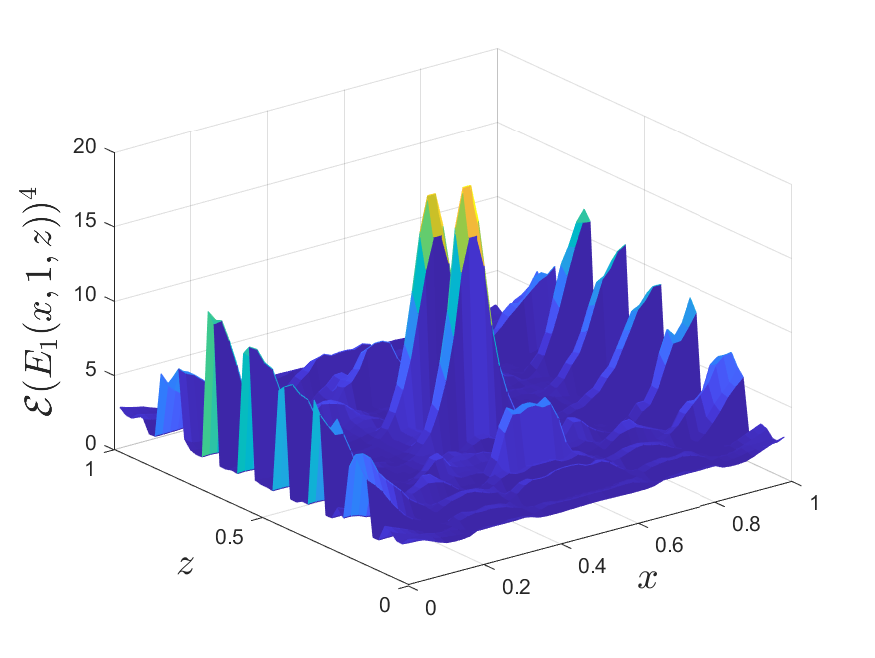}
  \includegraphics[height=3.3cm,width=4.5cm]{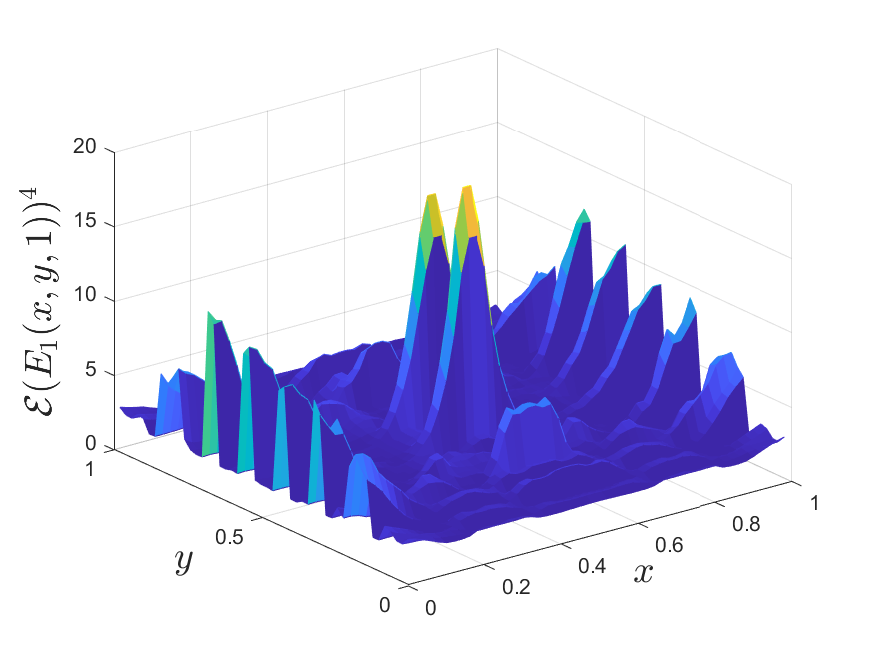}
  \includegraphics[height=3.3cm,width=4.5cm]{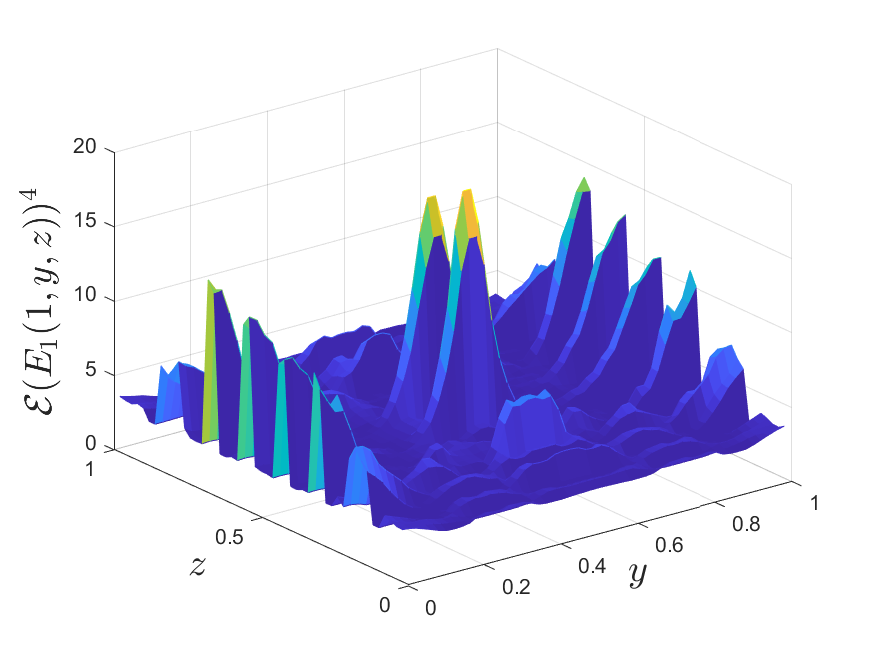}
    \includegraphics[height=3.3cm,width=4.5cm]{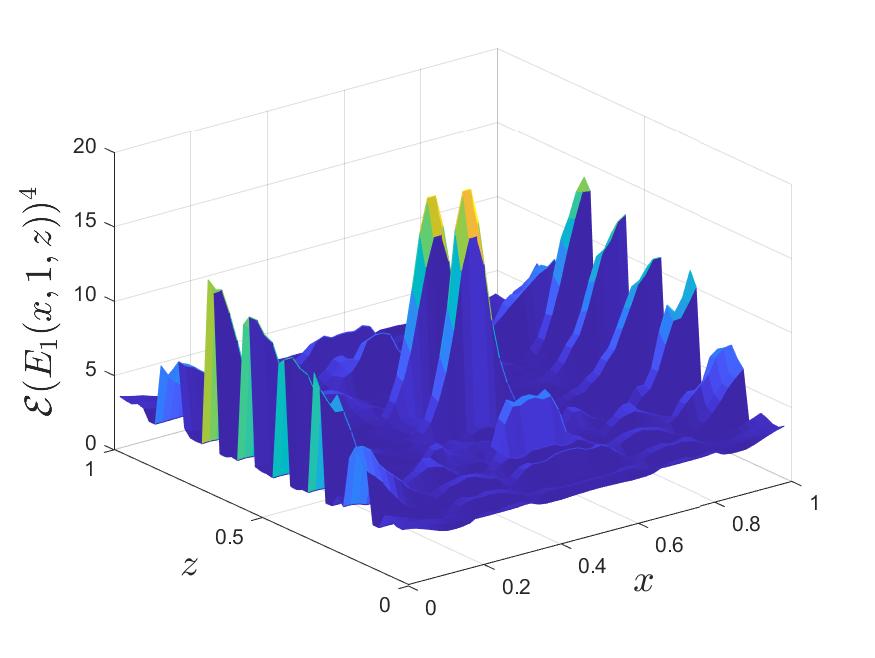}
  \includegraphics[height=3.3cm,width=4.5cm]{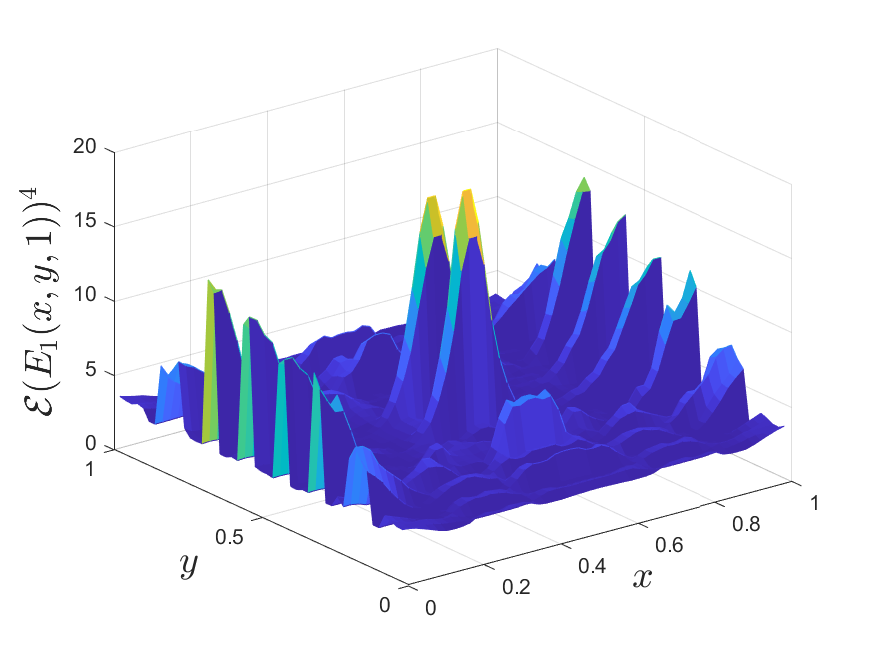}
  \caption{Fourth order statistical moments of $E_1$ when $x=1$, $y=1$ and $z=1$ 
  respectively by MC (top) and WCE (below).}\label{pic3:4.2}
  \end{center}
\end{figure} 
 \begin{table}[H]
\centering
\fontsize{10.5pt}{25pt}
 \begin{tabular}{cccccc} 
\toprule
 $u$ & $err\left[\mathcal{E}(u)\right]$& $ err\left[\mathcal{E}(u)^2\right]$ & $err\left[\mathcal{E}(u)^3\right]$ & $err\left[\mathcal{E}(u)^4\right]$\\
\midrule
$u=E_1$ & 0.0591 & 0.0314 & 0.0722  &0.0826\\
 \bottomrule
\end{tabular}
\caption{Relative errors of the WCE and MC}
\label{table3}
\end{table}

Finally, we compute the average energy for different $\sigma$. In order to accommodate the computational intensity, we take the spatial mesh grid size $\Delta x=\Delta y=\Delta z=1/40$ and the temporal step size is $\Delta t=1/1000$. Figure~\ref{pic3:4.3} presents the temporal evolution of the average energy. We can observe the results of WCE algorithm still are better than those of  MC method. Consistent with prior observations in 1-D and 2-D, the WCE algorithm demonstrates superior agreement with the exact solution compared to the  MC method, which exhibit minor deviations attributable to sampling variability.
\begin{figure}[H]
\begin{center}
  \includegraphics[height=4.4cm,width=8cm]{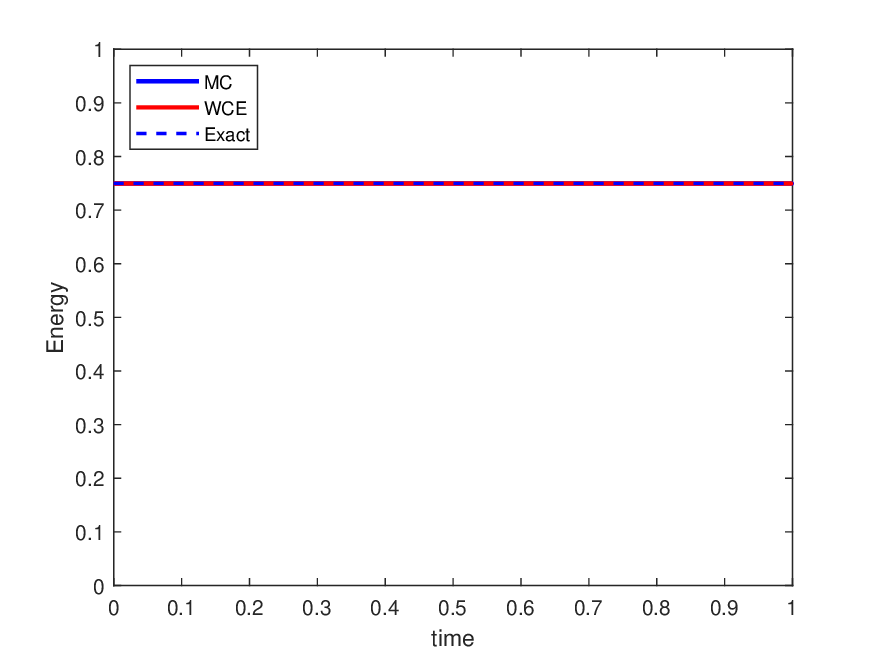}
  \includegraphics[height=4.4cm,width=8cm]{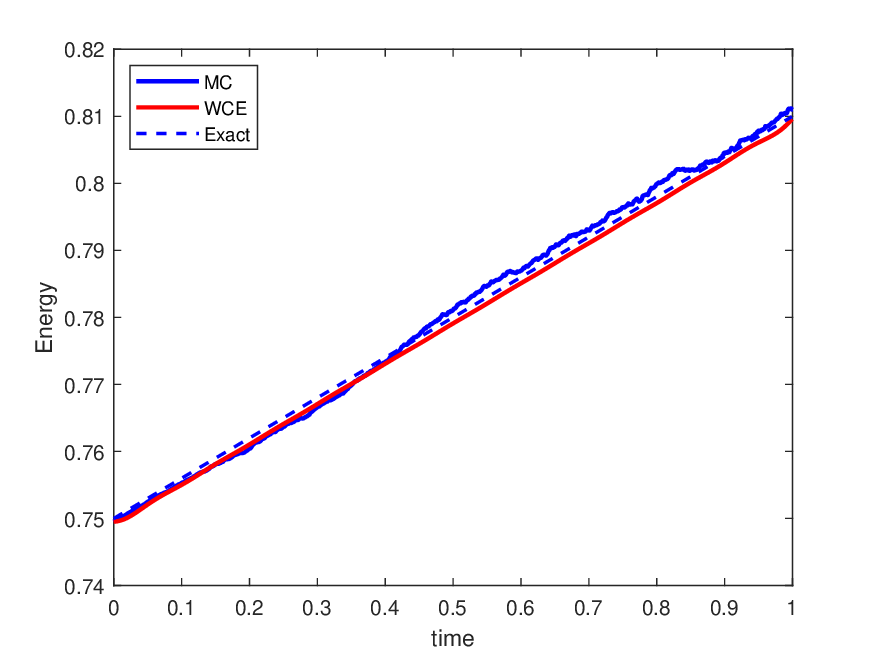}
  \includegraphics[height=4.4cm,width=8cm]{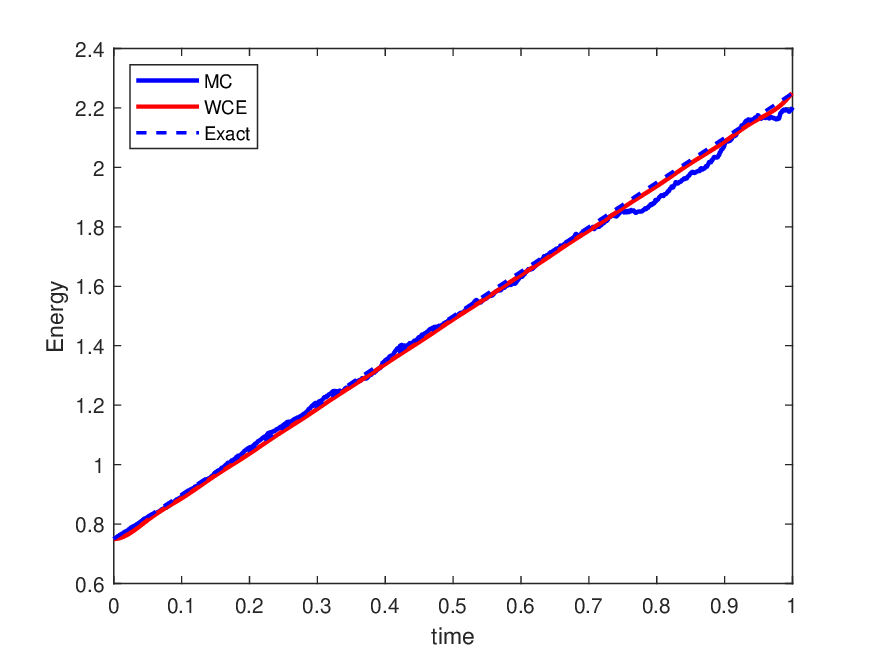}
  \includegraphics[height=4.4cm,width=8cm]{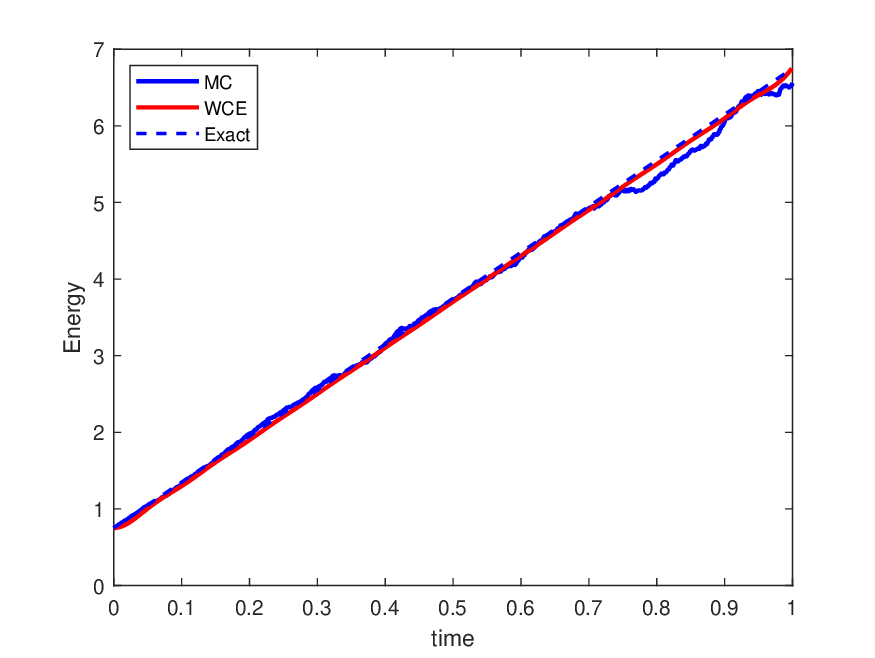}
  \caption{Averaged energy by WCE and MC for different sizes of noise $\sigma = 0$ (top left), $\sigma = 0.1$ (top right), $\sigma = 0.5$ (below left) and $\sigma = 1$ (below right).}\label{pic3:4.3}
  \end{center}
\end{figure}

\section{Concluding remarks}
In this paper, we propose a Wiener chaos expansion (WCE) algorithm for solving stochastic Maxwell equations with additive noise. By reformulating the stochastic Maxwell equations into a stochastic multi-symplectic structure, we derive a corresponding deterministic system for the WCE coefficients that inherits a similar multi-symplectic structure. The WCE framework also provides explicit formulations for computing the statistical moments and the averaged energy of the stochastic solution. Numerical experiments demonstrate that the WCE algorithm significantly outperforms the Monte Carlo (MC) method in both computational efficiency and accuracy when computing statistical moments. Furthermore, the results confirm that the averaged energy computed via the WCE algorithm adheres to the theoretical linear growth behavior.
\section*{Acknowledgments}
The authors would like to express their appreciation to the referees for their useful comments and the editors. Liying Zhang is supported by the National Natural Science Foundation of China (No. 11601514 and No. 11971458), the Fundamental Research Funds for the Central Universities (No. 2023ZKPYL02 and No. 2023JCCXLX01), the Yueqi Youth Scholar Research Funds for the China University of Mining and Technology-Beijing (No. 2020YQLX03) and the Fundamental Research Funds for the Central Universities (2025 Basic Sciences Initiative in Mathematics and Physics). Lihai Ji is supported by the National Natural Science Foundation of China (No. 12171047).


\begin{thebibliography}{}
 \bibitem{Luo} W.~Luo. Wiener chaos expansion and numerical solutions of stochastic partial differential equations, California Institute of Technology Pasadena, California, Ph.D. thesis, 2006.

\bibitem{Refhai}
J. Hong, L. Ji and L. Zhang, A stochastic multi-symplectic scheme for stochastic Maxwell equations with additive noise, J. Comput. Phys., \textbf{268} (2014) 255-268.

\bibitem{RefHJZC2017}
J. Hong, L. Ji, L. Zhang and J. Cai, An energy–conserving method for stochastic Maxwell equations with multiplicative noise, J. Comput. Phys., \textbf{351} (2017) 216-229.

\bibitem{RefKai}
K. Zhang, Numerical studies of some stochastic partial differential equations, Ph.D thesis, The Chinese University of Hong Kong, China, 2008.

\bibitem{RefCam}
R. Cameron and W. Martin, The orthogonal development of nonlinear functionals in series of Fourier-Hermite functionals, Ann. Math., \textbf{35} (1947) 385-392.

 \bibitem{RefCHJ2023}
 C. Chen, J. Hong and L. Ji, Numerical approximations
 of stochastic Maxwell equations via structure-preserving algorithms, Springer-Verlag, vol. 2341, 2023.

 \bibitem{RefRKT1989}
S. Rytov, Y. Kravtsov and V. Tatarskii, Principles of Statistical Radiophysics. 3: Elements of Random Fields, Springer-Verlag, 1989.

 \bibitem{RefChen2021}
C. Chen, A symplectic discontinuous Galerkin full discretization for stochastic Maxwell equations, SIAM J. Numer. Anal., \textbf{59} (2021) 2197-2217.

 \bibitem{RefCHJ2019a}
C. Chen, J, Hong and L. JI, Mean–square convergence of a semidiscrete scheme for stochastic Maxwell equations, SIAM J. Numer. Anal., \textbf{57} (2019) 728-750.

 \bibitem{RefCHJ2019b}
C. Chen, J, Hong and L. JI, Runge--Kutta semidiscretizations for stochastic Maxwell equations with additive noise, SIAM J. Numer. Anal., \textbf{57} (2019) 702-727.

 \bibitem{RefCCHS2020}
D. Cohen, J. Cui, J. Hong and L. Sun, Exponential integrators for stochastic Maxwell’s equations driven by It\^o noise, J. Comput. Phys., \textbf{410} (2020) 1-21.

 \bibitem{RefSSX2022}
J. Sun, C. Shu and Y. Xing, Multi-symplectic discontinuous Galerkin methods for the stochastic Maxwell equations with additive noise, J. Comput. Phys., \textbf{461} (2022) 111199.

 \bibitem{RefSSX2023}
J. Sun, C. Shu and Y. Xing, Discontinuous Galerkin methods for stochastic Maxwell equations
with multiplicative noise. ESAIM Math. Model. Numer. Anal., \textbf{57} (2023) 841-864.
 

 \bibitem{RefCHJL2025}
C. Chen, J. Hong, L. Ji and G. Liang, Invariant measures of stochastic Maxwell equations and
ergodic numerical approximations, J. Differ. Equations, \textbf{416} (2025) 1899-1959.

\bibitem{RefCHZ2016}
C. Chen, J. Hong and L. Zhang, Preservation of physical properties of stochastic Maxwell equations with additive noise via stochastic multi–symplectic methods, J. Comput. Phys., \textbf{306} (2016) 500-519.

\bibitem{RefZCHJ2019}
L. Zhang, C. Chen, J. Hong and L. Ji, A review on stochastic multi-symplectic methods for stochastic Maxwell equations, Commun. Appl. Math. Comput., \textbf{1} (2019) 467-501.

\bibitem{RefHLRZ2006}
T.Y. Hou, W. Luo, B. Rozovskii and H. Zhou, Wiener chaos expansions and numerical solutions of randomly forced equations of fluid mechanics, J. Comput. Phys., \textbf{216} (2006) 687-706.

\bibitem{RefMR2004}
R. Mikulevicius and B. L. Rozovskii, Stochastic Navier--Stokes equations for turbulence, SIAM J. Math. Anal., \textbf{35} (2004) 1250-1310.

\bibitem{RefZRTK2012}
Z. Zhang, B. Rozovskii, M.V. Tretyakov and G.E. Karniadakis, A multistage Wiener chaos expansion method for stochastic advection-diffusion-reaction equations, SIAM J. SCI. Comput., \textbf{34} (2012) A914-A936. 

\bibitem{RefBAZC2009}
M. Badieirostami, A. Adibi, H. Zhou and S. Chow, Wiener chaos expansion and simulation of electromagnetic wave propagation excited by a spatially incoherent source, Multiscale Model. Simul., \textbf{8} (2009) 591-604.

\end{thebibliography}
\end{document}